\nonstopmode
\documentclass[10pt]{amsart}
\usepackage{latexsym}
\usepackage{fancyhdr}
\usepackage{amsmath, amssymb,amsthm}
\usepackage[all]{xy}
\usepackage{pdflscape}
\usepackage{longtable}
\usepackage{rotating}
\usepackage{verbatim}
\usepackage{hyperref}
\usepackage{cleveref}
\usepackage{subfigure}
\usepackage{mathrsfs}
\usepackage{mdwlist}
\usepackage{dsfont}
\usepackage{bbold}
\usepackage{mathtools}
\usepackage{booktabs}
\usepackage{float}
\usepackage{color}

\newcounter{mnotecount}[section]
\renewcommand{\themnotecount}{\thesection.\arabic{mnotecount}}
\newcommand{\mnote}[1]
{\protect{\stepcounter{mnotecount}}$^{\mbox{\footnotesize $
\bullet$\themnotecount}}$ \marginpar{\color{cyan}%
\raggedright\tiny $\!\!\!\!\!\!\,\bullet$\themnotecount~ #1} }

\newcommand{\rmnote}[1]{}
\newcommand{\adr}[1]{\mnote{{\bf armin:} #1}}

\allowdisplaybreaks

\DeclareFontFamily{U}{mathb}{\hyphenchar\font45}
\DeclareFontShape{U}{mathb}{m}{n}{
      <5> <6> <7> <8> <9> <10> gen * mathb
      <10.95> mathb10 <12> <14.4> <17.28> <20.74> <24.88> mathb12
      }{}
\DeclareSymbolFont{mathb}{U}{mathb}{m}{n}
\DeclareFontSubstitution{U}{mathb}{m}{n}

\let\dot\relax
\DeclareMathAccent{\dot}{0}{mathb}{"39}
\let\ddot\relax
\DeclareMathAccent{\ddot}{0}{mathb}{"3A}
\let\dddot\relax
\DeclareMathAccent{\dddot}{0}{mathb}{"3B}
\let\ddddot\relax
\DeclareMathAccent{\ddddot}{0}{mathb}{"3C}

\theoremstyle{plain}
\newtheorem*{theorem*}{Theorem}
\newtheorem{theorem}{Theorem}[section]
\newtheorem*{lemma*}{Lemma}
\newtheorem{lemma}[theorem]{Lemma}
\newtheorem*{proposition*}{Proposition}
\newtheorem{proposition}[theorem]{Proposition}
\newtheorem*{corollary*}{Corollary}
\newtheorem{corollary}[theorem]{Corollary}
\newtheorem*{claim*}{Claim}
\newtheorem{claim}{Claim}

\newtheorem*{conjecture*}{Conjecture}

\newtheorem*{question*}{Question}
\newtheorem*{result*}{Result}
\newtheorem{result}[theorem]{Result}
\theoremstyle{definition}
\newtheorem*{definition*}{Definition}
\newtheorem{definition}[theorem]{Definition}
\newtheorem*{example*}{Example}
\newtheorem{example}[theorem]{Example}

\newtheorem*{algorithm*}{Algorithm}
\newtheorem*{remark*}{Remark}
\newtheorem*{remarks*}{Remarks}
\newtheorem{remark}[theorem]{Remark}

\newtheorem*{convention*}{Convention}

\numberwithin{equation}{section}

\def\al{\alpha}
\def\be{\beta}
\def\ga{\gamma}
\def\de{\delta}
\def\ep{\epsilon}

\def\la{\lambda}

\def\rh{\rho}

\def\si{\sigma}

\def\vh{\varphi}

\def\ps{\psi}
\def\om{\omega}

\def\Ga{\Gamma}

\def\Ps{\Psi}
\def\Om{\Omega}

\def\C{\mathbb{C}}

\def\N{\mathbb{N}}

\def\R{\mathbb{R}}

\def\cA{\mathcal{A}}

\def\cC{\mathcal{C}}

\def\cF{\mathcal{F}}

\def\cH{\mathcal{H}}

\def\sH{\mathscr{H}}

\def\sU{\mathscr{U}}

\def\p{\partial}

\def\<{\langle}
\def\>{\rangle}
\renewcommand{\o}{\circ}

\def\ol{\overline}
\def\interior{\on{int}}

\let\on=\operatorname

\newcommand{\sr}[1]%
{\ifmmode{}^\dagger\else${}^\dagger$\fi\ifvmode
\vbox to 0pt{\vss
 \hbox to 0pt{\hskip\hsize\hskip1em
 \vbox{\hsize3cm\raggedright\pretolerance10000
 \noindent #1\hfill}\hss}\vss}\else
 \vadjust{\vbox to0pt{\vss%
 \hbox to 0pt{\hskip\hsize\hskip1em%
 \vbox{\hsize3cm\raggedright\pretolerance10000%
 \noindent #1\hfill}\hss}\vss}}\fi%
}

\def\A{\;\forall}
\def\E{\;\exists}

\providecommand{\mapsfrom}{\kern.2em%
\setbox0=\hbox{$\leftarrow$\kern-.10em\rule[0.26mm]{0.1mm}{1.3mm}}\box0%
\kern.3em}

\title[Arc-smooth functions on closed sets]
{Arc-smooth functions on closed sets}

\author[A.~Rainer]{Armin Rainer}

\address{University of Education Lower Austria,
Campus Baden M\"uhlgasse 67, A-2500 Baden \&
Fakult\"at f\"ur Mathematik, Universit\"at Wien, 
Oskar-Morgenstern-Platz~1, A-1090 Wien, Austria}
\email{armin.rainer@univie.ac.at}

\begin{document}
	
\begin{abstract}
	By an influential theorem of Boman, a function $f$ on an open set $U$ in $\R^d$ is smooth ($\cC^\infty$) 
if and only if it is \emph{arc-smooth}, i.e., $f\o c$ is smooth for every smooth curve $c : \R \to U$. 
In this paper we investigate the validity of this result on closed sets. Our main focus is on sets 
which are the closure of their interior, so-called \emph{fat} sets. 
We obtain an analogue of Boman's theorem on fat closed sets with H\"older boundary 
and on fat closed subanalytic sets with the property that every 
boundary point has a basis of neighborhoods each of which intersects the interior in a 
connected set. 
If $X \subseteq \mathbb R^d$ is any such set and $f : X \to \R$ 
is arc-smooth, then $f$ extends to a smooth function defined on $\mathbb R^d$.
We also get a version of the Bochnak-Siciak theorem on all closed fat subanalytic and all closed 
sets  with H\"older boundary:
if $f : X \to \R$ is the restriction of a smooth function on $\R^d$ which is real analytic 
along all real analytic curves in $X$, 
then $f$ extends to a holomorphic function on a neighborhood of $X$ in $\mathbb C^d$.
Similar results hold for non-quasianalytic Denjoy-Carleman classes (of Roumieu type). 
We will also discuss sharpness and applications of these results.
\end{abstract}

\thanks{Supported by FWF-Project P~26735-N25}
\date{\today}

\maketitle

\tableofcontents

\section{Introduction}

In this paper we study differentiability of functions defined on closed subsets of $\R^d$. 
One way to endow an arbitrary set $X$ with a smooth structure is by declaring which curves $\R \to X$ and 
which functions $X \to \R$ should be \emph{smooth}. Together with a natural compatibility condition this leads to 
the notion of a \emph{Fr\"olicher space}; cf.\ \cite{FK88} and \cite{KM97}. 
Here we study the Fr\"olicher space generated by the inclusion of a closed set $X$ in $\R^d$ and some of its relatives. 
We will not use the terminology of Fr\"olicher spaces in the paper but the connection is made precise in 
\Cref{rem:Froelicher}.

\subsection{Boman's theorem and its relatives} \label{sec:intro}

Let $f : U \to \R$ be a function defined in an open subset $U$ of $\R^d$. 
Then $f$ induces a mapping $f_* : U^\R \to \R^\R$, $f_*(c) = f \o c$, whose invariance properties encode the 
regularity of $f$:

\begin{result}[Boman \cite{Boman67}] \label{result:1}
	A function $f : U \to \R$ is smooth ($\cC^\infty$) if and only if $f_* \cC^\infty(\R,U) \subseteq \cC^\infty(\R,\R)$.
\end{result}

Similarly, H\"older differentiability can be characterized by $f_*$; we denote by  
$\cC^{k,\al}$, for $k \in \N$, $\al \in (0,1]$, the class of $k$-times continuously differentiable 
	functions whose partial derivatives of order $k$ satisfy a local $\al$-H\"older condition.

\begin{result}[{\cite{FK88}, \cite{FF89}, \cite{KM97}}] \label{result:2}
	A function $f : U \to \R$ is of class $\cC^{k,\al}$ if and only if $f_* \cC^{\infty}(\R,U) \subseteq \cC^{k,\al}(\R,\R)$.
\end{result}

Furthermore, there is a ultradifferentiable version of Boman's theorem. 
We recall that, for a positive sequence $M=(M_k)_{k \in \N}$, 
the \emph{Denjoy--Carleman class} (of Roumieu type) $\cC^M(U, \R^m)$ is the set of 
all functions $f \in \cC^\infty(U,\R^m)$ such that for all compact $K \subseteq U$, 
\begin{equation} \label{DCcondition}
 	\E C,\rh >0  \A k \in \N \A x \in K : 
 	\|f^{(k)}(x)\|_{L_k(\R^d,\R^m)} \le C \rh^{k} k!\, M_{k}.  
\end{equation}
The sequence $M$ is called non-quasianalytic if $\cC^M$ contains non-trivial functions with compact support. 
If $M$ is log-convex, then $\cC^M$ is stable under composition. 
We refer to \Cref{sec:DenjoyCarleman} for this and more on Denjoy--Carleman classes.  

\begin{result}[\cite{KMRc}] \label{result:3}
	Assume that $M= (M_k)$ is non-quasianalytic and log-convex.	A function  
	$f : U \to \R$ is of class $\cC^{M}$ if and only if $f_* \cC^{M}(\R,U) \subseteq \cC^{M}(\R,\R)$.
\end{result}

\begin{remark} \label{rem:Boman}
Boman actually showed that $f$ is smooth if and only if $f_* \cC^M(\R,U) \subseteq \cC^\infty(\R,\R)$, 
for some arbitrary non-quasianalytic log-convex sequence $M$.	
\end{remark}

A glance at the proofs confirms that the curves along which the regularity in question is tested can be taken to 
have compact support.

A function $f : U \to \R$ with the property that $f \o c$ is real analytic ($\cC^\om$) for all real analytic $c : \R \to U$ 
clearly does not need to be real analytic on $U \subseteq \R^d$, let alone continuous, see \cite{BierstoneMilmanParusinski91}. 
But there is the following:

\begin{result}[{Bochnak, Siciak \cite{Bochnak70}, \cite{Siciak70}, \cite{BochnakSiciak71}}] \label{result:4}
	A function $f : U \to \R$ is real analytic if and only if 
	$f_* \cC^{\infty}(\R,U) \subseteq \cC^{\infty}(\R,\R)$ and $f_* \cC^{\om}(\R,U) \subseteq \cC^{\om}(\R,\R)$. 
\end{result}

Actually, a smooth function $f \in \cC^{\infty}(U)$ which is real analytic on affine lines is real analytic on $U$.

We remark that, if $M=(M_k)$ is quasianalytic such that $\cC^\om \subsetneq \cC^M$, 
then a $\cC^\infty$-function $f : U \to \R$ which satisfies 
$f_* \cC^{M}(\R,U) \subseteq \cC^{M}(\R,\R)$ need not be of class $\cC^M$; see \cite{Jaffe16}.

\subsection{Arc-smooth functions}

In this paper we investigate the validity of the above results on non-open subsets $X \subseteq \R^d$.
For arbitrary subsets $X\subseteq \R^d$ 
we define
\begin{align*}
	\cA^\infty(X) 
	&:= \big\{f : X \to \R : f_*\cC^\infty(\R,X) \subseteq \cC^\infty(\R,\R)\big\},
	\\
	\cA^M(X) &:= \big\{f : X \to \R : f_*\cC^M(\R,X) \subseteq \cC^M(\R,\R)\big\},
	\\
	\cA_M^\infty(X) &:= \big\{f : X \to \R : f_*\cC^M(\R,X) \subseteq \cC^\infty(\R,\R)\big\},
\end{align*}
where we set
\begin{align*}
	\cC^\infty(\R,X) &:= \big\{c \in \cC^\infty(\R,\R^d) : c(\R) \subseteq X\big\}, 
	\\
	\cC^M(\R,X) &:= \big\{c \in \cC^M(\R,\R^d) : c(\R) \subseteq X\big\}.	
\end{align*}
We call the elements of $\cA^\infty(X)$ \emph{arc-smooth functions} and those of $\cA^M(X)$ \emph{arc-$\cC^M$ functions} on $X$. 
We will also consider 
\begin{align*}
	\cA^\om(X) 
	&:= \big\{f \in \cA^\infty(X) : f_*\cC^\om(\R,X) \subseteq \cC^\om(\R,\R)\big\},
\end{align*}
where 
\[
	\cC^\om(\R,X) := \big\{c \in \cC^\om(\R,\R^d) : c(\R) \subseteq X\big\}.
\]
We will not speak of \emph{arc-analytic} functions, 
since such are not assumed to be smooth in the literature.

Evidently, $\cA^\om(X) \subseteq \cA^\infty(X) \subseteq \cA_M^\infty(X) \supseteq \cA^M(X)$. 
(We will see below that there is no hope for the analogue of \Cref{result:2} to hold on even very simple non-open sets 
like the closed half-space.)

With this notation, \Cref{result:1}, \Cref{result:3}, and \Cref{result:4} amount to 
\begin{equation} \label{open}
	\cA^{\infty}(X) = \cC^{\infty}(X),\quad \cA^{M}(X) = \cC^{M}(X),\quad \cA^{\om}(X) = \cC^{\om}(X),
\end{equation}
if $X \subseteq \R^d$ is a non-empty open set and $M=(M_k)$ is a non-quasianalytic log-convex sequence.

\begin{remark}
	The identities \eqref{open} imply that, in the definition of $\cA^\infty(X)$, $\cA^M(X)$, and $\cA^\om(X)$,  
	we could equivalently replace the families of curves $c  : \R \to X$ by 
	families of plots $p : U \to X$ (of the same regularity), where $U$ is any open subset of $\R^e$ with varying $e$. 
\end{remark}

\begin{remark} \label{rem:Froelicher}
	Recall that a \emph{Fr\"olicher space} is a triple $(X,\cC_X,\cF_X)$ consisting of a set $X$, a 
	subset $\cC_X \subseteq X^\R$ and a subset $\cF_X \subseteq \R^X$ such that 
	\begin{enumerate}
		\item $f : X \to \R$ belongs to $\cF_X$ if and only if $f \o c \in \cC^\infty(\R,\R)$ for all $c \in \cC_X$.
		\item $c : \R \to X$ belongs to $\cC_X$ if and only if $f \o c \in \cC^\infty(\R,\R)$ for all $f \in \cF_X$.  
	\end{enumerate}
	Any subset $\cF \subseteq \R^X$ generates a unique Fr\"olicher space $(X,\cC_X,\cF_X)$ by setting 
	\begin{align*}
	 	\cC_X &:= \big\{c : \R \to X : f \o c \in \cC^\infty(\R,\R) \text{ for all } f \in \cF\big\},
	 	\\
	 	\cF_X &:= \big\{f : X \to \R : f \o c \in \cC^\infty(\R,\R) \text{ for all } c \in \cC_X\big\}.
	 \end{align*} 
	 In this paper we are investigating the Fr\"olicher spaces generated by the inclusion map $\iota_X : X \to \R^d$ 
	 of subsets $X$ of $\R^d$, i.e., $(X,\cC^\infty(\R,X), \cA^\infty(X))$. 
	 For suitable sets $X$ we try to identify the corresponding set of functions $\cF_X = \cA^\infty(X)$.
	 More on Fr\"olicher spaces can be found in \cite{FK88} and \cite{KM97}. 
\end{remark}

\subsection{Admissible sets}
Let $X \subseteq \R^d$ be an arbitrary subset. 
A function $f : X \to \R$ is said to be \emph{smooth} if for each $x \in X$ there exist a neighborhood $U$ in $\R^d$ 
and a smooth function $F : U \to \R$ such that $F|_{U \cap X} = f|_{U \cap X}$. If $X$ is open, then this notion of 
smoothness coincides with the usual one. We denote by $\cC^\infty(X)$ the set of all smooth functions on $X$. 

\begin{definition} \label{def:admissible}
	A subset $X \subseteq \R^d$ is called \emph{$\cA^\infty$-admissible} if $\cA^\infty(X) = \cC^\infty(X)$, 
	i.e., the arc-smooth functions on $X$ are precisely the smooth functions.  
\end{definition}
 
Boman's theorem states that open subsets $X \subseteq \R^d$ are $\cA^\infty$-admissible. 
We will look for non-open $\cA^\infty$-admissible sets.
It follows from a result of Kriegl \cite{Kriegl97} that  
closed convex subsets $X \subseteq \R^d$ with non-empty interior are $\cA^\infty$-admissible. 
It is natural to consider closed sets with dense interior. 

\begin{definition}
	A non-empty closed subset $X$ of $\R^d$ is called \emph{fat} if $X = \ol {\interior(X)}$.	
\end{definition}

If $X \subseteq \R^d$ is fat, then there are other natural possibilities to define ``smooth'' functions on $X$ 
which we compare in the following lemma.

\begin{lemma} \label{lem:converse}
	Let $X \subseteq \R^d$ be a fat closed set. 
	Consider the following conditions:
	\begin{enumerate}
	 	\item There exists $F \in \cC^\infty(\R^d)$ such that $F|_X = f$.
	 	\item $f \in \cC^\infty(X)$.
	 	\item $f|_{\interior(X)} \in \cC^\infty(\interior(X))$ and the Fr\'echet derivatives $(f|_{\interior(X)})^{(n)}$ 
	 	of all orders have continuous extensions $f^{(n)} : X \to L_n(\R^d,\R)$. 
	 	\item $f|_{\interior(X)} \in \cC^\infty(\interior(X))$ and the directional derivatives $d_v^n f|_{\interior(X)}$ 
	 	for all $v \in \R^d$ and all $n \in \N$ have continuous extensions to $X$.
	 	\item $f|_{\interior(X)} \in \cC^\infty(\interior(X))$ and the partial derivatives $\p^\al f|_{\interior(X)}$ 
	 	for all $\al \in \N^d$ have continuous extensions to $X$.
	 \end{enumerate} 
	 Then $(1) \Rightarrow (2) \Rightarrow (3) \Leftrightarrow (4) \Leftrightarrow (5)$. 
	 All five conditions are equivalent if $X$ has the following regularity property:
	\begin{enumerate}
			\item[(6)]
			For all $x \in X$ there exist $m \in \N_{>0}$, $C>0$, and a compact neighborhood $K$ of $x$ in $X$ such that 
		any two points $y_1,y_2 \in K$ can be joined by a rectifiable path $\ga$ which lies in $\interior(X)$ except perhaps for finitely many points and has length
		\begin{equation*}
		 	\ell(\ga) \le C\, |y_1 - y_2|^{1/m}.
		 \end{equation*} 
		\end{enumerate}
\end{lemma}

\begin{proof}
	$(1) \Rightarrow (2) \Rightarrow (3)$ are obvious.

	$(3) \Leftrightarrow (4) \Leftrightarrow (5)$ 
	This follows from the fact that at points $x \in \interior(X)$ 
	Fr\'echet, directional, and partial derivatives can be converted into one another 
	in a linear way; cf.\ \cite[Lemma 7.13]{KM97}.

	$(5) \Rightarrow (1)$ By the regularity property (6), $f$ defines a Whitney jet on $X$, 
	see \cite[Proposition 2.16]{Bierstone80a}. 
	So Whitney's extension theorem implies (1).
\end{proof}

In general the implication $(5) \Rightarrow (1)$ is false, see \Cref{ex:complementflatcusp}.

Another natural condition for $\cA^\infty$-admissibility is the following; see \Cref{ex:simple}.

\begin{definition}
	\label{def:simple}
	A closed subset $X \subseteq \R^d$ is called \emph{simple} if each $x \in X$ has a basis of neighborhoods 
	$\sU$ such that $U \cap \interior(X)$ is connected for all $U \in \sU$. 
\end{definition}

A function $f : X \to \R$ is said to be \emph{real analytic} if for each $x \in X$ there exist a neighborhood $U$ 
of $x$ in $\C^d$ and a holomorphic function $F : U \to \C$ such that $F|_{U \cap X} = f|_{U \cap X}$. 
We denote by $\cC^\om(X)$ the set of all real analytic functions on $X$.  

If $M=(M_k)$ is a positive sequence, we set
\begin{equation*}
	\cC^M(X) := \big\{f \in \cC^\infty(X) :  \eqref{DCcondition} \text{ holds for all compact } K \subseteq X \big\}.
\end{equation*} 
Note that we do not require that a function $f \in \cC^M(X)$ is locally a restriction of a $\cC^M$-function on $\R^d$.
We shall discuss in \Cref{sec:extension} when a function in $\cC^M(X)$ extends to a $\cC^M$-function on $\R^d$.

\begin{definition}
	A subset $X \subseteq \R^d$ is called \emph{$\cA^\om$-admissible} 
(resp.\ \emph{$\cA^M$-admissible}) if $\cA^\om(X) = \cC^\om(X)$ (resp.\ $\cA^M(X) = \cC^M(X)$).
\end{definition}

By the Bochnak--Siciak theorem \ref{result:4} and \Cref{result:3}, all open subsets $X \subseteq \R^d$ are 
$\cA^\om$-admissible and 
$\cA^M$-admissible, for each log-convex non-quasianalytic $M$.

\subsection{Main results}

Our results can be arranged in groups with respect to two criteria: 
\emph{regularity of the functions} (smooth, real analytic, ultradifferentiable) 
and \emph{regularity of the domains} (H\"older sets, fat subanalytic sets).

By a \emph{H\"older set} we mean the \emph{closure} of an open set which has the uniform cusp property of index $\al$ 
for some $0<\al\le 1$. If $\al =1$ we speak of a \emph{Lipschitz set}. 
The collection of all H\"older sets in $\R^d$ is denoted by $\sH(\R^d)$.
(We use the term H\"older \emph{set} instead of \emph{domain}, 
since the latter is usually reserved for open sets.)
For precise definitions we refer to \Cref{sec:domains}.

\subsubsection*{The smooth case}

\begin{theorem} \label{main:1}
Every $X\in \sH(\R^d)$ 
is $\cA^\infty$-admissible.
We even have 
\begin{equation} \label{eq:smooth}
	\cA_M^\infty(X) =	\cA^\infty(X) = \cC^\infty(X),
\end{equation}  
for any non-quasianalytic log-convex positive sequence $M=(M_k)$.
\end{theorem}

\Cref{main:1} is proved in \Cref{sec:mainA}.

\begin{theorem} \label{main:4}
Every simple fat closed subanalytic set $X \subseteq \R^d$ is 
$\cA^\infty$-admissible.
\end{theorem}

This is proved in \Cref{sec:subanalytic}. 
The proof is based on the L-regular decomposition of subanalytic sets 
and the fact that fat closed subanalytic sets are uniformly polynomially 
cuspidal. It uses the result for H\"older sets, i.e., \Cref{main:1}. 

\begin{remark}
 H\"older sets $X \in \sH(\R^d)$ and fat closed subanalytic subsets $X \subseteq \R^d$ satisfy 
 \Cref{lem:converse}(6) and hence all items (1)--(5) in \Cref{lem:converse} are equivalent; cf.\ 
 \Cref{prop:alpharegular} and
	\Cref{thm:subanreg}. 	
\end{remark}

Notice that the assumption that $X$ is simple is necessary, see \Cref{ex:simple}.
H\"older sets are always simple, see \Cref{alpha-simple}.

\subsubsection*{The real analytic case}

\begin{theorem} \label{main:6}
Let $X \subseteq \R^d$ be a fat closed subanalytic set. 
Let $f \in \cC^\infty(X)$ 
		be real analytic on real analytic curves in $X$. 
		Then $f$ extends to a holomorphic function defined on an open neighborhood
		of $X$ in $\C^d$.
\end{theorem}

The proof of \Cref{main:6} (in \Cref{sec:BS}) is based on the uniformization theorem of subanalytic sets and 
a result of Eakin and Harris \cite{EakinHarris77} (proved earlier by Gabrielov \cite{Gabrielov73}). 
The following consequence will also be proved in \Cref{sec:BS}.

\begin{corollary} \label{main:7}
Let $X \subseteq \R^d$ be a closed set such that for all $z \in \p X$ there 
is a closed fat subanalytic set $X_z$ such that $z \in X_z \subseteq X$. 
Let $f \in \cC^\infty(X)$ 
		be real analytic on real analytic curves in $X$. 
		Then $f$ extends to a holomorphic function defined on an open neighborhood
		of $X$ in $\C^d$.
\end{corollary}

Note that all H\"older sets satisfy the assumption in \Cref{main:7}. 
Interestingly, for these results we need not assume that $X$ is simple 
(note that we already suppose that $f \in \cC^\infty(X)$).
Together with \Cref{main:1,main:4} we obtain:

\begin{corollary}
	Every $X \in \sH(\R^d)$ and
	every simple fat closed subanalytic $X \subseteq \R^d$ is 
	$\cA^\om$-admissible. \qed
\end{corollary}

\subsubsection*{The ultradifferentiable case}

Let $M = (M_k)$ be a non-quasianalytic log-convex positive sequence.
For positive integers $a$ let $M^{(a)}$ denote the sequence 
defined by $M^{(a)}_k := M_{ak}$.

\begin{theorem} \label{main:2}
Let $M = (M_k)$ be a non-quasianalytic log-convex positive sequence.
Every Lipschitz set $X \subseteq \R^d$ satisfies $\cC^M(X) \subseteq \cA^M(X) \subseteq \cC^{M^{(2)}}(X)$. 
\end{theorem}

A similar statement can be expected for H\"older sets (with the loss of regularity also depending on the 
H\"older index). We will not pursue this in this paper. Instead,
combining our results with 
a result of \cite{ChaumatChollet99} and \cite{BelottoBierstoneChow17}, 
we show in \Cref{main:5} that for fat closed subanalytic sets
 the loss of regularity can be controlled in a precise way.

	In an earlier version of the paper we claimed that every Lipschitz set $X \subseteq \R^d$ is 
	$\cA^M$-admissible. That is doubtful, but we do not have a counterexample.

\subsection{Permanence of admissibility}
The main results all concern subsets $X \subseteq \R^d$ with maximal dimension $d$. 
The following permanence properties yield further examples of admissible sets both of maximal dimension 
and of codimension $\ge 1$.

\begin{proposition} \label{prop:permanence}
	Let $X \subseteq \R^d$ be $\cA^\infty$-admissible. If $U$ is an open neighborhood of $X$ in $\R^d$ and 
	$\vh : U \to \R^e$ is a smooth embedding, then $\vh(X) \subseteq \R^e$ is $\cA^\infty$-admissible.  
\end{proposition}

\begin{proof}
	Let $Y := \vh(X)$. 
	If $f \in \cA^\infty(Y)$, then $g := f \o \vh \in \cA^\infty(X)$. 
	Since $M := \vh(U)$ is an embedded submanifold of $\R^e$, it suffices to show that 
	for each $y \in Y$ there is a neighborhood $V$ in $M$ and a smooth function $F : V \to \R$ such that 
	$F|_{V\cap Y} = f|_{V \cap Y}$.

	Since $X$ is $\cA^\infty$-admissible, for each $x \in X$ there is a neighborhood $W$ in $\R^d$ and a smooth function 
	$G : W \to \R$ such that $G|_{W \cap X} = g|_{W \cap X}$. Taking $U \cap W$ instead of $W$ we may assume that 
	$W \subseteq U$. 
	Then $F := G \o \vh^{-1}|_{\vh(W)}$ is smooth on $V := \vh(W)$ and satisfies 
	$F|_{V\cap Y} = f|_{V \cap Y}$. 
\end{proof}

The same proof yields the following.

\begin{proposition}
	Let $X \subseteq \R^d$ be $\cA^\om$-admissible. If $U$ is an open neighborhood of $X$ in $\R^d$ and 
	$\vh : U \to \R^e$ is a real analytic embedding, then $\vh(X) \subseteq \R^e$ is $\cA^\om$-admissible.  \qed
\end{proposition}

In the ultradifferentiable case we have the following. 
Note that, if $M=(M_k)$ is log-convex, then $\cC^M$ is stable under composition and 
the $\cC^M$ inverse function theorem holds. 
If $N \subseteq \R^e$ is an embedded submanifold of class $\cC^M$ (i.e., the chart change maps are of class $\cC^M$), 
then we define $\cC^M(N)$ to be the set of $f \in \cC^\infty(N)$ which are of class $\cC^M$ in every local coordinate 
chart. 
If $Y \subseteq N$, then let $\cC^M(Y)$ be the set of $\cC^\infty$-functions on $Y$ such that the defining 
estimates hold for all compact subsets in $Y$ in all local coordinate charts. 
The proof of \Cref{prop:permanence} implies the following.

\begin{proposition}
	Let $M= (M_k)$ be non-quasianalytic and log-convex, and let $N=(N_k)$ be a sequence with $M \le N$.	
	Assume that $X \subseteq \R^d$ satisfies $\cC^M(X) \subseteq \cA^M(X) \subseteq \cC^N(X)$. 
	If $U$ is an open neighborhood of $X$ in $\R^d$ and 
	$\vh : U \to \R^e$ is a $\cC^M$-embedding, then $Y:=\vh(X) \subseteq \R^e$ satisfies 
	$\cC^M(Y) \subseteq \cA^M(Y) \subseteq \cC^N(Y)$. 
	\qed 
\end{proposition}

\subsection{Sharpness of the results}
We shall discuss in \Cref{counterexamples} counterexamples which show that none of the conditions in the 
main results can in general be omitted without suitable replacement.  

In particular, \Cref{flat}, which is based on a division theorem of \cite{JorisPreissmann90}, shows
that the \emph{$\infty$-flat cusp} $$X := \big\{(x,y) \in \R^2 : x\ge 0,\, 0 \le y \le \exp(-1/x)\big\}$$ 
is not $\cA^\infty$-admissible: in this case $\cA^\infty(X)$ is strictly larger 
than $\cC^\infty(X)$. 
Note, however, that for $Y := \R^2 \setminus \interior(X)$ we have $f \in \cA^\infty(Y)$ if and only if $f$ 
satisfies \Cref{lem:converse}(3), but $\cA^\infty (Y) \ne \cC^\infty(Y)$; see 
\Cref{ex:complementflatcusp}. 

Interestingly, the analogue for finite differentiability (i.e., \Cref{result:2}) fails even on convex fat closed sets 
such as the half-space; see \Cref{ex:Glaeser} 
which is a consequence of Glaeser's inequality.

\subsection{Applications} 
As a corollary of the real analytic result (i.e., \Cref{main:6}) we obtain that smooth solutions of real analytic equations 
on H\"older sets or closed fat subanalytic sets must be real analytic; see \Cref{thm:aneq}.
Furthermore, we obtain sufficient conditions for the existence of real analytic solutions $g$ of the 
equation $f = g \o \vh \in \cC^\om(M)$, where $\vh : M \to \R^d$ is a real analytic map defined on a real analytic 
manifold $M$; see \Cref{fep}. 

The usefulness of the smooth result is illustrated by some consequences for the division of smooth functions, 
see \Cref{thm:division}, and for pseudo-immersions, see \Cref{thm:pseudoimmersion}.  

\subsection{Structure of the paper} 
We recall facts on weight sequences and Denjoy--Carleman classes in \Cref{sec:curvelemma}, 
and we revisit and adapt the $\cC^M$ curve lemma which is an essential tool for proving some results of the paper. 
In \Cref{sec:domains} we introduce H\"older sets and collect some of their properties. 
The proofs of \Cref{main:1,main:4,main:6,main:2} are given in the Sections \ref{sec:mainA}, \ref{sec:subanalytic}, 
\ref{sec:BS}, and \ref{sec:mainB},
respectively.   
In \Cref{ssec:ultra} we discuss the ultradifferentiable case on subanalytic sets.  
The applications are given in \Cref{sec:applications}. 
The final \Cref{sec:complements} contains complements, examples, and counterexamples.

Some of the results of this paper were announced in \cite{Rainer17}.

\subsection*{Acknowledgements}
I am grateful to Vincent Grandjean, Andreas Kriegl, and Adam Parusi\'nski for helpful discussions.
A.\ Kriegl contributed \Cref{lem:shrink} and \Cref{ex:independent}
and A.\ Parusi\'nski suggested to use the results on H\"older sets to attack subanalytic sets.
In addition, I would like to thank the anonymous referees for their valuable comments.

\section{A \texorpdfstring{$\cC^M$}{CM}-curve lemma} \label{sec:curvelemma}

This section is only of relevance for the ultradifferentiable results in the paper.

\subsection{Weight sequences and Denjoy--Carleman classes} \label{sec:DenjoyCarleman}

Let $M = (M_k)_{k\in \N}$ be a positive sequence of reals. 
Let $U \subseteq \R^d$ be open and let $\cC^M(U,\R^m)$ be the corresponding   
Denjoy--Carleman class (of Roumieu type) as defined in \Cref{sec:intro}.

If $N=(N_k)$ is another positive sequence such that $(M_k/N_k)^{1/k}$ is bounded, then 
$\cC^M(U) \subseteq \cC^N(U)$. 
The converse holds if $k! M_k$ is logarithmically convex (\emph{log-convex} for short).
It follows that the class $\cC^M(U)$ is preserved by replacing $M=(M_k)_k$ by $(C^k M_k)_k$ for some positive 
constant $C$.

We shall assume that
the sequence $M$ is log-convex (which entails log-convexity of $k! M_k$).
We may assume that $M_0 = 1$ and that $M$ is increasing.
Indeed, the sequence $N_k:=C^k M_k/M_0$ for some constant $C \ge M_0/M_1$, 
is log-convex, increasing, satisfies $N_0 = 1$, and $\cC^M(U) = \cC^N(U)$. This motivates the following definition. 

\begin{definition}
	An increasing log-convex sequence $M = (M_k)$ with $M_0 =1$ is called a \emph{weight sequence}. 
	
\end{definition}

The regularity properties of a weight sequence $M=(M_k)$ entail stability properties of the class 
$\cC^{M}$; cf.\ \cite{RainerSchindl14}.
Of particular interest in this paper is the fact that, for a weight sequence $M$,
the composite of $\cC^M$ mappings is $\cC^M$.
By the celebrated Denjoy--Carleman theorem, the condition
\begin{equation} \label{nq}
 	\sum_{k} \frac{M_k}{(k+1) M_{k+1}}  < \infty
 \end{equation} 
holds if and only if $\cC^M$ is non-quasianalytic, i.e., 
the Borel mapping which sends germs at some point $a$ 
of smooth functions 
to their infinite Taylor expansion at $a$ is not injective on $\cC^M$-germs. 
Then there exist non-trivial $\cC^M$-functions with compact support.
Note that \eqref{nq} is equivalent to 
\begin{equation} \label{nq2}
 	\sum_{k} (k!\, M_k)^{-1/k}  < \infty.
\end{equation}

\begin{definition}
	Let $M=(M_k)$ be a weight sequence.
	We say that $M$ is \emph{non-quasianalytic} if it satisfies \eqref{nq}; otherwise it is said to be \emph{quasianalytic}. 
	A weight sequence $M$ is called \emph{strongly non-quasianalytic} if 
\begin{equation} \label{snq}
	\E C> 0 \A k \in \N : \sum_{j\ge k} \frac{M_{j-1}}{j M_j} \le C \frac{M_{k-1}}{M_k}.
\end{equation}
It is said to be of \emph{moderate growth} if 
\begin{equation} \label{mg}
	\E C > 0 \A j,k \in \N : M_{j+k} \le C^{j+k} M_j M_k.
\end{equation}
A weight sequence is called \emph{strongly regular} if it is strongly non-quasianalytic and of moderate growth.
\end{definition}

\begin{example}
	The Gevrey sequences $G^s_k = k!^s$, $s >0$, which give rise to the Gevrey classes $\cC^{G^s}$ are 
	strongly regular weight sequences. They appear naturally in the theory of (partial) differential equations. 
	For $s=0$ we recover the real analytic functions $\cC^{G^0} = \cC^\om$ which obviously form a quasianalytic class.   	
\end{example}

Note that $\cC^\om(U) \subseteq \cC^M(U) \subseteq \cC^\infty(U)$ for every weight sequence $M$. 
In fact, the Denjoy--Carleman classes form an a scale of spaces intermediate between the real analytic and 
the smooth functions.

\subsection{The \texorpdfstring{$\cC^M$}{CM} curve lemma revisited}

We generalize the $\cC^M$ curve lemma, see \cite[Section 3.6]{KMRc} and \cite[Section 2.5]{KMRq}, which was inspired by 
\cite[Lemma 2]{Boman67}. 

\begin{lemma} \label{curvelemma}
	There are sequences $t_k \to t_\infty$ and $s_k>0$ in $\R$ with the following property.
	For any non-quasianalytic weight sequence $M=(M_k)$ and each $a \in \N_{\ge 2}$ there is a real positive sequence 
	$\la_k \to 0$ satisfying 
	\begin{equation} \label{laM}
		\la_k \Big(\frac{M_{ak}}{M_k}\Big)^{\frac1{ak+1}} \to 0 \quad \text{ as } k \to \infty
	\end{equation}
	such that the following holds.
	Let $E$ be a Banach space.
	Let $c_k \in C^{\infty}(\R,E)$ be a sequence such that 
	\begin{equation} \label{eq:lamdaconverge}
		\big\{\la_k^{-1}c_k^{(\ell)}(t) : t \in I,\, \ell,k \in \N\big\}
	\end{equation}
	is 
	bounded in $E$, for every bounded interval $I \subseteq \R$. 
	Then there exists a $\cC^M$-curve $c : \R \to E$ with compact support and 
	$c(t_k+t) = c_k(t)$ for $|t| \le s_k$.	
\end{lemma}

\begin{proof}
	There exists a non-quasianalytic weight sequence $L = (L_k)$ such that $(M_k/L_k)^{1/k} \to \infty$
	(this follows e.g.\ from \cite[Lemma 6]{Komatsu79b}). 
	Choose a $\cC^L$-function $\vh: \R \to [0,1]$ which is $0$ on $\{t : |t| \ge 1/2\}$ and $1$ 
	on $\{t : |t| \le 1/3\}$.   

	Let $T \in (0,1]$ and $R >0$. Assume that $\ga \in C^{\infty}(\R,E)$  
	is such that
	\[
		\| \ga^{(\ell)}(t) \| \le R \quad \text{ for all } |t| \le 1/2,\, \ell \in \N.
	\] 
	Then, there exist $C,\rh \ge 1$ such that for the curve $c(t) := \vh(t/T) \ga(t)$ we have
	\begin{align} \label{eq:CMcurve}
		\|c^{(\ell)}(t)\| &= \Big\|\sum_{j=0}^\ell \binom{\ell}{j} T^{-j} \vh^{(j)}\Big(\frac t T\Big) \ga^{(\ell-j)}(t) \Big\|
		\\
		&\le R \sum_{j=0}^\ell  \binom{\ell}{j} T^{-j} C \rh^{j} j! L_{j} \notag
		\le C R  \Big(1+\frac{\rh}{T}\Big)^{\ell} \ell! L_{\ell}
		\le C R  \Big(\frac{2\rh}{T}\Big)^{\ell} \ell! L_{\ell}   .
	\end{align}
	
	Choose a sequence 
	\begin{equation} \label{eq:Tt}
		\text{$T_j \in (0,1]$ with $\sum_j T_j < \infty$ and let $t_k := 2 \sum_{j<k} T_j + T_k$.}
	\end{equation}
	Now choose $\la_j$ such that the following conditions are fulfilled: 
	\begin{align}
		&0 < \frac{\la_j}{T_j^k} \le \frac{M_k}{L_k} \quad \text{ for all }  j,k, \label{la1}
		\\
		&\frac{\la_j}{T_j^k} \to 0 \quad \text{ as }  j \to \infty \text{ for all } k. \label{la2}
	\end{align}
	It suffices to take $\la_j \le \inf_{k} T_j^{k+1} M_k/L_k$.
	Clearly, we may in addition require that $\la_j$ tends to zero fast enough so that \eqref{laM} holds.

	By \eqref{eq:lamdaconverge}, there is $R >0$ such that   
	\[
		\|c_k^{(\ell)}(t)\| \le R \la_k \quad \text{ for all }  |t|\le 1/2,\, \ell,k \in \N.
	\] 
	Define 
	\[
		c(t) := \sum_j \vh\Big(\frac{t-t_j}{T_j}\Big) c_j(t-t_j).
	\]
	The summands have disjoint supports (the support of the $j$th summand is contained in $[t_j-T_j/2, t_j +T_j/2]$). 
	Thus $c$ is $\cC^\infty$ on $\R \setminus \{t_\infty\}$. By \eqref{eq:CMcurve},
	\[
		\|c^{(\ell)}(t)\| \le C R  \la_j  \Big(\frac{2\rh}{T_j}\Big)^{\ell} \ell! L_{\ell}   \quad \text{ for } |t-t_j| 
		\le \frac{T_j}{2}. 
	\]
	Consequently, by \eqref{la1},
	\[
		\|c^{(\ell)}(t)\| \le CR   (2\rh)^{\ell} \ell!  M_{\ell}     
		\quad \text{ for } t \ne t_\infty. 
	\]
	It follows that $c : \R \to E$ has compact support and is $\cC^M$ (cf.\ \cite[Lemma 2.9]{KM97} and \cite[Lemma 3.7]{KMRc}).
\end{proof}

\begin{remark}
	A similar statement holds for convenient vector spaces $E$. 
	The proof can be easily adapted to this case; cf.\ \cite{KMRc} or \cite{KMRq}.  
\end{remark}

The next lemma is a variant of \cite[Lemma 2.8]{KM97}.
Recall that, given some sequence $\mu_k \to \infty$, a sequence $x_k$ in $E$ is called \emph{$\mu$-convergent} to $x$ 
if $\mu_k (x_k - x)$ is bounded.

\begin{lemma} \label{special}
	For any non-quasianalytic weight sequence $M=(M_k)$ there is a positive sequence $\la_k \to 0$ 
	such that the following holds.
	Let $E$ be a Banach space.
	Let $x_n \to x$ be $1/\la_k$-convergent in $E$.  
	Then the infinite polygon through the $x_n$ and $x$ can be parameterized as a $\cC^M$-curve 
	$c : \R \to E$ such that $c(1/n) = x_n$ and $c(0) = x$.
\end{lemma}

\begin{proof}
	Let $L=(L_k)$ be a non-quasianalytic weight sequence with $(M_k/L_k)^{1/k} \to \infty$.
	Set $T_j := 1/(j(j+1))$ and choose $\la_j$ such that the conditions \eqref{la1} and \eqref{la2}
	are satisfied. 
	Let $\vh : \R \to [0,1]$ be a $\cC^L$-function which vanishes on $(-\infty,0]$ and is $1$ on $[1,\infty)$. 
	Let $t_n:= 1/n$ and define 
	\[
		c(t) := \begin{cases}
			x & \text{ if } t \le 0,
			\\
			x_{n+1} + \vh\Big(\frac{t - t_{n+1}}{t_{n} - t_{n+1}}\Big) (x_n - x_{n+1}) & \text{ if }  
			t_{n+1} \le t \le t_{n}
			\\
			x_1 & \text{ if } t\ge 1.
		\end{cases}
	\]
	Clearly, $c$ is $\cC^\infty$ on  $\R \setminus \{0\}$.
	For $t_{n+1} \le t \le t_{n}$ we have 
	\begin{align*}
		c^{(k)}(t) &=  \vh^{(k)} \Big(\frac{t - t_{n+1}}{t_{n} - t_{n+1}}\Big) 
		 (n(n+1))^k (x_n - x_{n+1})
		\\
		&=  \vh^{(k)} \Big(\frac{t - t_{n+1}}{t_{n} - t_{n+1}}\Big) 
		\cdot  \frac{\la_n}{T_n^k} \cdot \frac{x_n - x_{n+1}}{\la_n}.
	\end{align*}
	Condition \eqref{la2} guarantees that $c^{(k)}(t) \to 0$ as $t \to 0$ for all $k$, and hence $c$ is $\cC^\infty$ on $\R$. 
	That $c$ is of class $\cC^M$ follows from \eqref{la1}.
\end{proof}

\section{H\"older sets}
\label{sec:domains}

\subsection{Uniform cusp property and H\"older sets}

We denote by $B(x,\ep) := \{y \in \R^d : |x-y| < \ep\}$ the open ball with center $x$ and radius $\ep$ in $\R^d$.

\begin{definition}[Truncated open cusp] \label{def:cusp}
	Let us consider $\R^d = \R^{d-1} \times \R$ with the Euclidean coordinates $x = (x_1,\ldots,x_d) = (x',x_d)$.  
	For $0< \al \le 1$ and $r,h>0$, consider the \emph{truncated open cusp}
	\[
		\Ga_d^{\al}(r,h) 
		:= \big\{(x',x_d) \in \R^{d-1} \times \R : |x'| < r ,\, h (|x'|/r)^{\al} < x_d < h\big\}.
	\]
	For $\al=1$ this is a \emph{truncated open cone}.   
\end{definition}

\begin{definition}[Uniform cusp property]
	Let $U \subseteq \R^d$ be an open set and 
	let $\al \in (0,1]$.
	We say that $U$ has the \emph{uniform cusp property of index $\al$}   
	if
	for every $x \in \p U$ there exists $\ep>0$, a truncated open cusp $\Ga = \Ga_d^{\al}(r,h)$, and an 
	orthogonal linear map $A \in \on{O}(d)$ such that 
	for all $y \in \ol U \cap B(x,\ep)$ we have $y + A\Ga \subseteq U$.  
\end{definition}

\begin{definition}[H\"older set]
	By an \emph{$\al$-set} we mean a closed fat set $X\subseteq \R^d$ such that $\interior(X)$ has the 
	uniform cusp property of index $\al$. 
	We say that $X \subseteq \R^d$ is a \emph{H\"older set} if it is an $\al$-set for some $\al \in (0,1]$.
\end{definition}

We denote by $\sH^\al(\R^d)$ the collection of all $\al$-sets in $\R^d$ and by 
	\[
	\sH(\R^d) := \bigcup_{0<\al\le 1} \sH^\al(\R^d)
	\]
	the collection of all H\"older sets in $\R^d$.
Note that $\sH^\al(\R^d) \subseteq \sH^\be(\R^d)$ if $\al \ge \be$. 

\begin{remark}
A bounded open subset $U \subseteq \R^d$ has the uniform cusp property of index $\al$ 
if and only if $U$ has H\"older boundary of index $\al$ with 
uniformly bounded H\"older constant; see \cite[Theorem 6.9, p.116]{DelfourZolesio11} and \cite[Theorem 1.2.2.2]{Grisvard85}. 
That means the following.
At each boundary point $p$ there is an orthogonal system of coordinates $(x',x_d)$ and an $\al$-H\"older function $a = a(x')$ 
such that in a neighborhood of $p$ 
the boundary of $U$ is given by $\{x_d = a(x')\}$ and the set $U$ is of the form $\{x_d > a(x')\}$. 
There is a uniform bound for the H\"older constant of $a$ which is independent of the boundary point $p$. 	
\end{remark}

	The boundary of an $\al$-set with $\al<1$ can be quite irregular.
	It may have Hausdorff dimension strictly larger than $d-1$ and hence its Hausdorff measure $\cH^{d-1}$ 
	may be locally infinite. See \cite[Theorem 6.10, p.~116]{DelfourZolesio11}.

\begin{example} \label{ex:alphadomain}
	(1) The set $X = \{(x,y) \in \R^2 : x\ge 0, \,  |y| \le  x^{1/\al}\}$ is an $\al$-set.

	(2) The set $X = \{(x,y) \in \R^2 : x\ge 0, \,  x^2 \le y \le  2x^{2}\}$ is not a H\"older set, 
	but $X$ is the image of the H\"older set $\{(x,y) \in \R^2 : x\ge 0, \,  |y| \le  x^{2}/2\}$ under the 
	diffeomorphism $(x,y) \mapsto (x,y+3x^2/2)$ of $\R^2$. 

	(3) The set $X = \{(x,y) \in \R^2 : x\ge 0, \, x^{3/2} \le y \le 2 x^{3/2}\}$ is not a H\"older set 
	and there is no smooth diffeomorphism of $\R^2$ which maps $X$ to a H\"older set. 

	(4) Let $C \subseteq [0,1]$ be the Cantor set and let $f : [0,1] \to \R$ be defined by 
	$f(x) := \on{dist}(x,C)^\al$. Then the set $X = \{(x,y) \in \R^2 : -1\le x \le 2,\,  f(x)\le y \le 2 \text{ if } x \in [0,1], \,
	0\le y \le 2 \text{ if } x \not\in [0,1]\}$ is an $\al$-set. 
\end{example}

\subsection{\texorpdfstring{$c^\infty$}{c-infinity}-topology on H\"older sets}

The \emph{$c^\infty$-topology} on a locally convex space $E$ is the final topology with respect to all 
smooth curves $c : \R \to E$. 
The $c^\infty$-topology on $\R^d$ coincides with the usual topology; cf.\ \cite[Theorem 4.11]{KM97}.
The \emph{$c^\infty$-topology} on a subset $X \subseteq E$ is the final topology with respect to all 
smooth curves $c : \R \to E$ satisfying $c(\R) \subseteq X$.

\begin{proposition} \label{topology}
	Let $X \in  \sH(\R^d)$. 	
	Then the $c^\infty$-topology of $X$ coincides with the trace topology from $\R^d$. 
\end{proposition}

\begin{proof}
Let $A\subseteq X$ be $c^\infty$-closed in $X$.
	Let $\ol A$ be the closure of $A$ in $\R^d$. 
	We have to show that $\ol A \cap X  = \ol A \subseteq A$.
	The converse implication is obvious.

	Let $x \in \ol A$. 
	Then there is a sequence $x_n \in A$ which tends to $x$. 
	It suffices to find a smooth curve $c \in \cC^\infty(\R,X)$ passing through 
	a subsequence of $x_n$ and through $x$. Since $A$ is $c^\infty$-closed in $X$, this shows 
	$x \in A$.

	Since $X$ is an $\al$-set, for some $0< \al \le 1$,
	we may assume that 
	there is a neighborhood $U$ of $x$ in $X$ and a cusp 
	$\Ga = \Ga_d^{\al}(r,h)$
	such that for all $y \in U$ we have $y+\Ga \subseteq \interior(X)$.
	By rescaling, we may assume that $r=h=1$.

	Consider $C(y,r) := y + \Ga_d^{\al}(r,r^\al)$ for $0<r\le 1$. It is easy to see that
	there is a universal constant $c>0$ such that $C(y_1,r_1) \cap C(y_2,r_2) \ne \emptyset$
	provided that $|y_1-y_2| \le c \min\{r_1,r_2\}$. 

	Choose a decreasing sequence $\mu_n$ which tends to $0$ faster than any polynomial. 
	By passing to a subsequence of $x_n$ (again denoted by $x_n$), we may assume that 
	$|x - x_n| \le c \mu_{n+1}/2$ for all $n$. Then, for all $n$,
	\[
		|x_n - x_{n+1}| \le |x- x_n| + |x- x_{n+1}| \le c \mu_{n+1}.
	\]
	Setting $C_n:= C(x_n,\mu_n)$
	this guarantees the existence of a sequence $u_n$ such that $u_{n+1} \in C_n \cap C_{n+1}$ for all $n$.
	By construction, $x_n$ and $u_n$ tend to $x$ faster than any polynomial.

	\begin{figure}[ht]
		\includegraphics[scale=0.3]{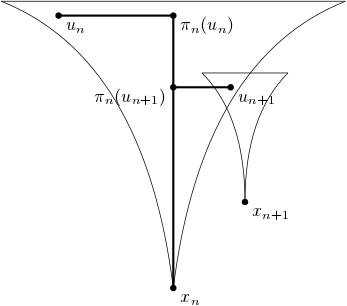}
	\end{figure}

	For $u \in C_n$ define $\pi_n(u) := x_n + u_d e_d$ (where $\{e_i\}$ is the standard basis in $\R^d$). 
	Consider the polygon $P_n$ through the points $u_n$, $\pi_n(u_n)$, $x_n$, $\pi_n (u_{n+1})$, $u_{n+1}$. 
	It is contained in $C_n$. 
	The infinite polygon consisting of the concatenation of all $P_n$ satisfies the assumptions of 
	\cite[Lemma 2.8]{KM97} 
	and can hence by parameterized by a smooth curve $c$ which is contained in $X$ and satisfies $c(0) =x$.
\end{proof}

\begin{remark} \label{rem:topology}
It is not difficult to modify the proof  
in order to obtain the following:
{\it 
Let $X \in \sH(\R^d)$ and let $M=(M_k)$ be a non-quasianalytic weight sequence. 	
	Then the final topology on $X$ with respect to all $\cC^M$-curves $c : \R \to \R^d$ with $c(\R) \subseteq X$ 
	coincides with the trace topology from $\R^d$.} 	
It suffices to take $\mu_n := \la_n^{1/\al}$ for the sequence $\la_n$  
provided by \Cref{special}.
\end{remark}

\subsection{Further properties of H\"older sets} 

The following proposition is well-known. We include a proof for the convenience of the reader.

\begin{proposition} \label{prop:alpharegular}
	Let $X\in \sH^\al (\R^d)$.
	Then for each $x \in X$ there is a compact neighborhood $K$ of $x$ in $X$ and a constant $D>0$ 
	such that 
	any two points $y_1,y_2 \in K$ can be joined by a polygon $\ga$ contained in $K$ with 
	$\p X \cap \ga \subseteq \{y_1,y_2\}$ of length 
	\begin{equation*} 
		\ell(\ga) \le D |y_1-y_2|^\al.
	\end{equation*}
\end{proposition}

\begin{proof}
	Clearly each $x \in \interior(X)$ has this property. 
	Let $x \in \p X$.
	We may assume that in a compact neighborhood $K$ of $x$ the set 
	$X$ is the epigraph $\{x_d \ge  f(x')\}$ of a $\al$-H\"older function $f$ 
	with respect to an orthogonal system of coordinates $(x',x_d) = (x_1,\ldots,x_d)$.
	For two points $y_1,y_2 \in K$ consider the segments $S:= [y_1,y_2]$ and $S':= [y_1',y_2']$. 
	If $(y_1,y_2) \subseteq K \cap \interior(X)$ there is nothing to prove.
	Otherwise let $z' \in S'$ be such that 
	$f(z') = \max_{y' \in S'} f(y')$ and let $z = (z',z_d)$ with $z_d := f(z') + \ep |y_1 - y_2|$ 
	for some small $\ep>0$ such that $z \in K \cap \interior(X)$.
	It is possible to choose $\ep$ such that it only depends on $K$, not on $y_1,y_2$. 
	We have $(y_i)_d \le  f(z')$ and 
	thus $|(y_i)_d - f(z')| \le |f(y_i') - f(z')|$ for at least one $i \in \{1,2\}$, say for $i=1$. 
	If $(y_2)_d \le  f(z')$, then
	the polygon with vertices $y_1$, $(y_1',z_d)$, $(y_2',z_d)$, $y_2$ is contained in $K$, meets $\p X$ 
	at most at one of the points $y_i$, and has
	length
	\begin{align*}
		|(y_1)_d - z_d| &+  |(y_2)_d - z_d| + |y_1' - y_2'| 
		\\
		&\le |f(y_1') - f(z')| + |f(y_2') - f(z')| + 2\ep |y_1 - y_2| + |y_1' - y_2'| 
		\\
		&\le C |y_1' - z'|^\al + C|y_2' - z'|^\al + (1+2\ep) |y_1 - y_2|
		\\
		&\le D |y_1 - y_2|^\al,
	\end{align*}
	for constants only depending on $K$.
	If $(y_2)_d >  f(z')$, then the segment joining $z$ and $y_2$ is contained in $K \cap \interior(X)$, 
	and thus the polygon with vertices $y_1$, $(y_1',z_d)$, $z$, $y_2$ is contained in $K$, meets $\p X$ 
	at most at one of the points $y_i$, and has
	length
	\begin{align*}
		|(y_1)_d - z_d| &+  |y_2 - z| + |y_1' - z'| 
		\\
		&\le |f(y_1') - f(z')| + (1+\ep) |y_1 - y_2| + |y_1' - y_2'| 
		\\
		&\le D |y_1 - y_2|^\al. 
	\end{align*}
	This finishes the proof.
\end{proof}

\begin{proposition} \label{alpha-simple}
	Every $X \in \sH(\R^d)$ is simple in the sense of \Cref{def:simple}.
\end{proposition}

\begin{proof}
	The proof of \Cref{prop:alpharegular} implies that there is a basis of neighborhoods $\sU$ of each $x \in X$ 
	such that $\interior(X) \cap U$ is path-connected for each $U \in \sU$. 
\end{proof}

\section{Arc-smooth functions on H\"older sets} \label{sec:mainA}

The aim of this section is to prove \Cref{main:1}:
\emph{All $X \in \sH(\R^d)$ are $\cA^\infty$-admissible. 
	We even have  
	\begin{equation} \label{eq:alpha-domainM}
		\cA_M^\infty(X)=\cA^\infty(X) = \cC^\infty(X),
	\end{equation}
	for any non-quasianalytic weight sequence $M=(M_n)$.}

\begin{remark}
For fat closed convex sets $X \subseteq \R^d$, $\cA^\infty$-admissibility follows from a result of  
Kriegl \cite{Kriegl97}. 
The statement in \cite{Kriegl97} is more general:
{\it Let $X$ be a convex subset of a convenient vector space $E$ with non-empty interior. Then 
$f \in \cA^\infty(X)$ if and only if $f$ is smooth on $\interior(X)$ and all 
Fr\'echet derivatives $(f|_{\interior(X)})^{(n)}$
extend continuously to $f^{(n)} : X \to L_n(E,\R)$ with respect to the $c^\infty$-topology of $X$.} 
In general the $c^\infty$-topology is finer than the given locally convex topology.	
\end{remark}

\subsection{Proof of \texorpdfstring{\Cref{main:1}}{Theorem 4.1}}

It is evident that 
\[
	\cC^\infty(X) \subseteq \cA^\infty(X) \subseteq \cA_M^\infty(X).
\]
The second inclusion is by definition, the first inclusion is a simple consequence of the chain rule. 
Let us prove the other inclusions.

\begin{lemma} \label{lem:Faa}
	Let $1 \le p\le q$ be integers. 
	For $x \in \R^d$ and $v = (v',v_d) \in \R^d$ 
	let $c(t) = x + (t^{q} v', t^{p}v_d)$, for $t$ in a neighborhood of $0 \in \R$. Let 
	$f$ be of class $\cC^{q}$ in a neighborhood of the image of $c$. 
	Then:
	\begin{align*}
		\frac{(f\o c)^{(k)}(0)}{k!} = 
			\begin{cases}
				\frac{1}{j!} f^{(j)}(x)((0,v_d)^j) & \text{ if } k= jp < q,
				\\ 
				f'(x)((v',0)) & \text{ if } k =q \not\in p\N,
				\\
				f'(x)((v',0)) + \frac{1}{j!} f^{(j)}(x)((0,v_d)^j) & \text{ if } k = jp=q.
			\end{cases}
	\end{align*}
	For all other $k<q$ we have $(f\o c)^{(k)}(0)=0$. 
\end{lemma}

\begin{proof}
	This follows easily from an inspection of the Fa\`a di Bruno formula 
	\[
		\frac{(f \o c)^{(k)}(t)}{k!} =  \sum_{j\ge 1} \sum_{\substack{\al_i>0 \\ \al_1 + \cdots + \al_j = k}} 
		\frac{f^{(j)}(c(t))}{j!} 
		\Big(\frac{c^{(\al_1)}(t)}{\al_1!}, \ldots, \frac{c^{(\al_j)}(t)}{\al_j!}\Big)
	\]
	and the special form of $c$.
\end{proof}

\begin{proposition}  \label{Krieglcusp} 
	Let $X \in  \sH(\R^d)$  
	and $f \in \cA^\infty(X)$. Then $f|_{\interior(X)}$ is smooth
	and its derivative $(f|_{\interior(X)})'$ extends uniquely to a mapping 
	$f' : X \to L(\R^d,\R)$ which belongs to $\cA^\infty(X, L(\R^d,\R))$, i.e.,
	\begin{equation} \label{concl}
		(f')_* \cC^\infty(\R,X) \subseteq \cC^\infty(\R,L(\R^d,\R)).
	\end{equation}
\end{proposition}

\begin{proof}
	That $f|_{\interior(X)}$ is smooth follows from Boman's theorem \ref{result:1}.

	There is $0 < \al \le 1$ such that $X \in \sH^\al(\R^d)$.
	Let $x \in \p X$. 
	We may assume that there is a truncated open cusp $\Ga = \Ga_d^{\al}(r,h)$ 
	and an open neighborhood $Y$ of $x$ in X
	such that for all 
	$y \in Y$ we have $y + \Ga \subseteq \interior (X)$. 	
	It suffices to show that $(f|_{Y \cap \interior (X)})'$ extends uniquely to a mapping 
	$f' : Y \to L(\R^d,\R)$ which belongs to $\cA^\infty(Y, L(\R^d,\R))$.

	Let $p < q$ be positive integers such that $p/q \le \al$ and $q/p \not\in \N$.
	Let $x \in Y$ and $v = (v',v_d) \in  \Ga$.  
	Then the curve 
	\[
		c_{x,v}(t) := x+ (t^{2q} v',t^{2p} v_d)
	\]
	lies in $\interior (X)$ for $0<|t|<1$ and 
	$c_{x,v}(0) = x$.  		 
	Since $f \in \cA^\infty(X)$, $f \o c_{x,v}$ is $\cC^\infty$.
	
	Let $v \in \Ga$ be fixed. We define 
	\[
		f'(x)(v) := \frac{(f \o c_{x,v})^{(2p)}(0)}{(2p)!} + \frac{(f \o c_{x,v})^{(2q)}(0)}{(2q)!}, \quad \text{ for } x \in Y.
	\]
	This definition turns into a correct statement if $x \in \interior(X)$, by \Cref{lem:Faa}.

	We claim that 
	\begin{equation} \label{smooth}
			\text{$f'(\cdot)(v) :  Y \to \R$ maps $\cC^\infty$-curves to $\cC^\infty$-curves.}
	\end{equation}
	Let $\R \ni s \mapsto x(s)$ be a $\cC^\infty$-curve in $Y$.  
	Then 
	$(s,t) \mapsto c_{x(s),v}(t)$ is a smooth mapping near $(0,0)$ with values in $X$. 
	Thus $(s,t) \mapsto f(c_{x(s),v}(t))$ is smooth, by Boman's theorem \ref{result:1}. So, in particular, 
	$s \mapsto \p_t^k|_{t=0}(f\o c_{x(s),v}(t))$ is smooth for all $k$. 
	It follows that  
	$s \mapsto  f'(x(s))(v)$ is smooth, which implies the claim.

	Let $s \mapsto x(s)$ be any $\cC^\infty$-curve in $Y$ such that $x(s) \in \interior(X)$ for $0<|s| \le 1$ and $x(0)=x_0$.
	Then 
	\begin{align*}
		f'(x_0)(v) &= \frac{(f \o c_{x_0,v})^{(2p)}(0)}{(2p)!} + \frac{(f \o c_{x_0,v})^{(2q)}(0)}{(2q)!}
		\\
		&= \lim_{s\to 0} \Big(\frac{(f \o c_{x(s),v})^{(2p)}(0)}{(2p)!} + \frac{(f \o c_{x(s),v})^{(2q)}(0)}{(2q)!} \Big)
		= \lim_{s\to 0} f'(x(s))(v).
	\end{align*}
	Consequently, the given definition of $f'(x_0)(v)$ is the only possible 
	extension of $f'(\cdot)(v)$ to $x_0$ which is continuous on $\cC^\infty$-curves.

	Now let $v \in \R^d$ be arbitrary. 
	Since $\Ga$ is open, there exist $\ep>0$ and $w \in \Ga$ such that $\ep v+w \in \Ga$.
	For all $x \in Y \cap \interior(X)$, we have
	\[
		f'(x)(v) =   \frac{f'(x)(\ep v + w) - f'(x)(w)}{\ep},
	\]
	and the right-hand side uniquely 
	extends to points $x_0 \in Y \cap \p X$ and satisfies \eqref{smooth}, by the arguments above.

	Thus, we define $f'(x_0)(v) := \lim_{s \to 0} f'(x(s))(v)$ for some $\cC^\infty$-curve $s \mapsto x(s)$ in $Y$ 
	with $x(0) = x_0$ and $x(s) \in \interior(X)$ for $0<|s|\le 1$.
	Then $f'(x_0)$ is linear as the pointwise limit of $f'(x(s)) \in L(\R^d,\R)$.
	The definition does not depend on the curve $x$, since it is the unique extension for $v \in  \Ga$.

	Let us finally show that $f' : Y \to L(\R^d,\R)$ belongs to $\cA^\infty(Y,L(\R^d,\R))$. 
	Let $x : \R \to Y$ be a $\cC^\infty$-curve and let $v \in \R^d$. 
	It suffices to show that $s \mapsto f'(x(s))(v)$ is smooth. 
	For $v \in  \Ga$ this follows from \eqref{smooth}. 
	For general $v$, $f'(x(s))(v)$ is a linear combination of 
	$f'(x(s))(v_1)$ and $f'(x(s))(v_2)$ for $v_i \in  \Ga$
	which locally is independent of $s$. The proof is complete.
\end{proof}

\begin{corollary}
	\label{cor:Krieglcusp} 
	Let $M=(M_k)$ be a non-quasianalytic weight sequence.
	Let $X \in \sH(\R^d)$  
	and $f \in \cA_M^\infty(X)$.
	Then $f|_{\interior(X)}$ is smooth
	and its derivative $(f|_{\interior (X)})'$ extends uniquely to a mapping 
	$f' : X \to L(\R^d,\R)$ which belongs to $\cA_M^\infty(X, L(\R^n,\R))$, i.e.,
	\begin{equation} \label{conclMinfty}
		(f')_* \cC^M(\R,X) \subseteq \cC^\infty(\R,L(\R^d,\R)).
	\end{equation}
\end{corollary}

\begin{proof}
	The proof is the same with the only difference that we use $\cC^M$-curves (thanks to \Cref{rem:Boman}); 
	note that the curves $c_{x,v}$ are polynomial 
	and thus of class $\cC^M$. 
\end{proof}

\begin{proof}[Proof of \Cref{main:1}]
	Let $f \in \cA^\infty(X)$ (resp.\ $f \in \cA_M^\infty(X)$). 
	\Cref{Krieglcusp} and \Cref{cor:Krieglcusp} imply by induction that 
	the Fr\'echet derivatives $(f|_{\interior(X)})^{(n)}$ of all orders have unique extensions 
	$f^{(n)} : X \to L_n(\R^d,\R)$ which satisfy 
	\begin{equation*}
		(f^{(n)})_* \cC^\infty(\R,X) \subseteq \cC^\infty(\R,L_n(\R^d,\R))
	\end{equation*}
	(resp.\ $(f^{(n)})_* \cC^M(\R,X) \subseteq \cC^\infty(\R,L_n(\R^d,\R))$).
	So $f$ satisfies \Cref{lem:converse}(3),  	
	since the $c^\infty$-topology of $X$ (resp.\ the final topology 
	on $X$ with respect to all $\cC^M$-curves in $X$) 
	coincides with the trace topology from $\R^d$, by \Cref{topology} (resp.\  \Cref{rem:topology}). 
	Thus $f \in \cC^\infty(X)$, by \Cref{lem:converse} and \Cref{prop:alpharegular}. 
\end{proof}

 \section{Arc-smooth functions on subanalytic sets} \label{sec:subanalytic}

The goal of this section is to prove \Cref{main:4}.

\subsection{Subanalytic sets}

Let $M$ be a real analytic manifold.
A subset $X$ of $M$ is called \emph{subanalytic} if for each $x \in M$ there is an open neighborhood $U$ of $x$ 
in $M$ such that $X \cap U$ is the projection of a relatively compact semianalytic subset of $M \times N$, 
where $N$ is a real analytic manifold. Recall that a subset $X$ of a real analytic manifold $M$ is semianalytic 
if  for each $x \in M$ there exist an open neighborhood $U$ of $x$ 
in $M$ and finitely many real analytic functions $f_{ij},g_{ij}$ on $U$ 
such that 
\[
X \cap U = \bigcup_i \bigcap_j \big\{f_{ij} = 0,\, g_{ij}>0\big\}.
\]
If $\dim M \le 2$, then the family of subanalytic sets in $M$ coincides with the family of semianalytic sets.
In higher dimensions the family of subanalytic sets is essentially larger.

Henceforth we restrict to the case $M= \R^d$.

 \begin{theorem}[Rectilinearization of subanalytic sets \cite{Hironaka73}, \cite{BM88}, \cite{Parusinski94}] 
\label{rectilinearization}
	Let $X \subseteq \R^d$ be closed subanalytic.  
	There exists a locally finite collection of real analytic mappings $\vh_\al : U_\al \to \R^d$ 
	such that each $\vh_\al$ is the composite of a finite sequence of local blow-ups with smooth centers and 
	\begin{enumerate}
		\item each $U_\al$ is diffeomorphic to $\R^d$ and there are compact subsets $K_\al \subseteq U_\al$ 
		such that $\bigcup_\al \vh_\al(K_\al)$ is a neighborhood of $X$ in $\R^d$,
		\item for each $\al$, $\vh_\al^{-1}(X)$ is a union of quadrants in $\R^d$. 
	\end{enumerate}
\end{theorem}

A \emph{quadrant} in $\R^d$ is a set 
\[
	Q(I_0,I_-,I_+) = \big\{ x \in \R^d : x_i = 0 \text{ if } I_0, \, x_i \le 0 
	\text{ if } I_-, \, x_i \ge 0 \text{ if } I_+\big\},
\]
where $I_0$, $I_-$, $I_+$ is any partition of $\{1,2,\ldots,d\}$.

\subsection{Bounded fat subanalytic sets are uniformly polynomially cuspidal} \label{sec:PP}
 
This is due to Paw{\l}ucki and Ple\'sniak \cite{PawluckiPlesniak86}. 
We recall some steps of the proof which will be needed later.

\begin{definition} \label{def:UPC}
A subset $X \subseteq \R^d$ is called \emph{uniformly polynomially cuspidal} 	
	if there exist positive constants $M,m>0$ and a positive integer $n$ such that 
	for all $x \in \ol X$ there is a polynomial curve $h_x : \R \to \R^d$ of degree at most $n$ with 
	the following properties:
	\begin{enumerate}
			\item $h_x((0,1]) \subseteq X$ and $h_x(0)=x$,
			\item $\on{dist}(h_x(t),\R^d \setminus X) \ge Mt^m$ for all $x \in X$ and all $t \in (0,1]$. 
		\end{enumerate}		
\end{definition}

\begin{remark}  
	Every compact set $X \in \sH(\R^d)$ is 
	uniformly polynomially cuspidal; this is clear by 
	\Cref{def:UPC}.  
	The converse is not true: 
	for instance, the sets in \Cref{ex:alphadomain}(2) and (3) are uniformly polynomially cuspidal but not in $\sH(\R^d)$. 
	The set $X$ in \Cref{ex:complementflatcusp} is uniformly polynomially cuspidal but neither subanalytic nor in $\sH(\R^d)$;
	cf.\ \cite[p.\ 284]{PawluckiPlesniak88}.
\end{remark}

\begin{theorem}[{\cite[Proposition 6.3]{PawluckiPlesniak86}}] \label{thm:PP}
	Let $X$ be a bounded open subanalytic subset of $\R^d$. 
	Then there is a map $h : \ol X \times \R \to \R^d$ such that $h(x,t)$ is a polynomial in $t$ with degree $n$
	independent of $x \in \ol X$ with $h(x,0) = x$ for all $x \in \ol X$,  
	$h(\ol X \times (0,1]) \subseteq X$, and 
	there exist positive constants $M,m$ such that 
	\[
		\on{dist}(h(x,t),\R^d \setminus X) \ge M t^m, \quad \text{ for all } x \in \ol X,\, t \in [0,1]. 
	\]
\end{theorem}

We give a sketch of the proof in order to explicate the uniformity of $h_x$ which will be of 
importance later.

The following is a corollary of the rectilinearization theorem.

\begin{proposition}[{\cite[Proposition 6.3]{PawluckiPlesniak86}}] \label{cor:PP}
	Let $X$ be a relatively compact subanalytic subset of $\R^d$ of pure dimension $d$.
	Then there is a finite number of real analytic maps $\vh_j : \R^d \times \R \to \R^d$ 
	such that, for $I^d := [-1,1]^d$,  
	\begin{gather*}
		\vh_j(I^d \times (0,1]) \subseteq X \quad \text{ for all }j,
		\\ 
		\bigcup_j \vh_j(I^d \times \{0\}) = \ol X.
	\end{gather*}
\end{proposition}

Let $X$ be a bounded open subanalytic subset of $\R^d$. 
Let $\vh_j$ be the maps provided by \Cref{cor:PP}. 
Then, for each $j$, the function 
\[
	I^d \times [0,1] \ni (y,t) \mapsto \on{dist}(\vh_j(y,t), \R^d \setminus X)	
\] 
is subanalytic (cf.\ \cite[Remark 3.11]{BM88}). 
By the {\L}ojasiewicz inequality (cf.\ \cite[Theorem 6.4]{BM88}), there exist positive constants $L,m$ such that 
\[
	\on{dist}(\vh_j(y,t), \R^d \setminus X) \ge L t^{m}, \quad  (y,t) \in I^d\times [0,1]. 
\] 
The constants $L$, $m$ may be assumed independent of $j$, by taking the minimum and maximum, respectively. 
Choose an integer $n \ge m$ and write
\[
	\vh_j(y,t)  = T_j(y,t) + t^{n+1} Q_j(y,t), \quad (y,t) \in \R^d \times \R,  
\]
where $T_j(y,\cdot)$ is the Taylor polynomial at $0$ of degree $n$ of $\vh_j(y,\cdot)$ and $Q_j : \R^d \times \R \to \R^d$ 
is real analytic. If we choose $\de \in (0,1]$ such that 
$|tQ_j(y,t)| \le L/2$ for all $j$, $y \in I^d$, and $t \in [0,\de]$, then
\[
	\on{dist}(T_j(y,t), \R^d \setminus X) \ge L t^m - \frac{L}{2} t^n \ge \frac{L}{2} t^m, \quad (y,t) \in I^d \times [0,\de]. 
\]
Replacing $t$ by $\de t$, we obtain
\[
	\on{dist}(T_j(y,\de t), \R^d \setminus X) \ge  M t^m, \quad (y,t) \in I^d \times [0,1], 
\]
where $M:= L \de^m/2$. Clearly, $\bigcup_j T_j(I^d \times \{0\}) = \bigcup_j \vh_j(I^d \times \{0\}) = \ol X$. 
\Cref{thm:PP} follows.

\subsection{Fat closed subanalytic sets are \texorpdfstring{$m$}{m}-regular}

Another property of fat closed subanalytic sets we need is the fact that they are \emph{$m$-regular} 
in the following sense.

\begin{theorem}[{\cite[Theorem 6.17]{Bierstone80a}, \cite{Hardt83}, \cite[Theorem 6.10]{BM88}}] \label{thm:subanreg} 
	Let $X \subseteq \R^d$ be a fat closed subanalytic set. 
	For each $a \in X$ there exist a compact neighborhood $K$ in $X$, a constant $C>0$, 
	and a positive integer $m$ such that any two points 
	$x,y \in K$ can be joined by a semianalytic path $\ga$ in $X$ which intersects $\p X$ in 
	at most finitely many points and satisfies 
	\[
		\ell(\ga) \le C \,|x-y|^{1/m}.
	\]
\end{theorem}

\subsection{L-regular decomposition}
Let us recall the \emph{L-regular decomposition} of subanalytic sets.

First we introduce sets which are \emph{L-regular} with respect to a given system of coordinates.
Let $X\subseteq \R^d$ be a subanalytic set of dimension $d$.
If $d =1$, then $X$ is called L-regular, if $X$ is a non-empty compact interval.
If $d >1$, then $X$ is L-regular, if it is of the form
\begin{equation} \label{eq:Lreg1}
	X= \big\{(x',x_d) \in \R^d : f(x') \le x_d \le g(x'), \, x \in X'\big\},
\end{equation}
where $X' \subseteq \R^{d-1}$ is L-regular and $f$, $g$ are continuous subanalytic functions on $X'$,
analytic and satisfying $f<g$ on $\interior (X')$ with bounded partial derivatives of first order.
If $\dim X = k < d$, then $X$ is L-regular, if 
\begin{equation} \label{eq:Lreg2}
	X = \big\{(y,z) \in \R^k \times \R^{d-k} : z= h(y),\, y \in X'\big\},
\end{equation}
where $X' \subseteq \R^k$ is L-regular, $\dim X'=k$, 
and 
$h$ is continuous subanalytic on $X'$,
analytic on $\interior (X')$ with bounded partial derivatives of first order.

In general a subanalytic set $X$ in $\R^d$ is said to be \emph{L-regular}, if it is L-regular with respect to 
some linear (or equivalently orthogonal) system of coordinates.
It is called an \emph{L-regular cell}, if it is the relative interior of an L-regular set, i.e.,
it is $\interior (X)$
in case \eqref{eq:Lreg1} and
the graph of $h$ restricted to $\interior (X')$ in case \eqref{eq:Lreg2}. 
By definition, every point is a zero-dimensional L-regular cell. 

It is well-known that L-regular sets and L-regular cells are \emph{quasiconvex} (cf.\ \cite{Kurdyka92},
\cite[Lemma 2.2]{Parusinski94L}, or \cite{KurdykaParusinski06}): 
there is a constant $C>0$ such that
any two points $x,y$ in the set can be joined in the set by a subanalytic path of length $\le C\, |x-y|$.

\begin{theorem}[{\cite{Kurdyka92}, \cite{KurdykaParusinski06}, \cite{Pawlucki08}}]\label{thm:Lreg}  
	Let $X \subseteq \R^d$ be a bounded subanalytic set.
	Then $X$ is a finite disjoint union of L-regular cells.
\end{theorem}

For the proof of \Cref{main:4}
we need the following preparatory results.

\begin{lemma} \label{lem:sequenceofcells}
	Let $[a,b] \subseteq \R$ be a non-trivial interval such that $[a,b] = \bigcup_{i=1}^k F_i$ for 
	closed sets $F_i$. If $a \le \sup F_i < b$ then there exists $j \ne i$ such that $\sup F_i \in F_j$ 
	and $\sup F_i < \sup F_j$. 
\end{lemma}

\begin{proof}	
	Fix $i$ and 
	suppose that $t := \sup F_i < b$. There is a sequence $(t,b) \ni t_n \to t$. 
	After passing to a subsequence we may assume that $t_n \in F_j$ for some fixed $j \ne i$. 
	Since $F_j$ is closed, $t \in F_j$. 
\end{proof}

\begin{lemma} \label{lem:quasiconvex}
	Let $X \subseteq  \R^d$ be a fat closed subanalytic set. Let $x \in \p X$ and suppose there is a 
	basis of neighborhoods $\sU$ of $x$ such that $U \cap \interior(X)$ is connected for all $U \in \sU$. 
	Then there is $U_0 \in \sU$ and a positive constant $C$ such that the following holds.
	For all $U \in \sU_0 := \{U \in \sU : U \subseteq U_0\}$ and 
	for any two points $y,z \in U \cap \interior(X)$, there exists a 
	rectifiable path $\ga$ in $\interior(X)$ which connects $y$ and $z$ and satisfies
	\[
		\ell(\ga) \le C  \on{diam}(U).
	\] 
\end{lemma}

\begin{proof}
	We may assume that $X$ is bounded, by intersecting with a ball centered at $x$. 
	Let $U_0$ be any member of $\sU$ which is contained in this ball.
	By \Cref{thm:Lreg}, $\interior (X)$ is a finite disjoint union of L-regular cells $\{A_1,\ldots,A_k\}$. 

	Fix $U \in \sU_0$ and let $y,z \in U \cap \interior(X)$.
	Since $U \cap \interior(X)$ is 
	connected, there is a path $\si : [0,1] \to U \cap \interior(X)$ with $\si(0) = y$ and $\si(1) = z$. 
	Then we have a finite disjoint union $[0,1] = \bigcup_{i=1}^k E_i$, where $E_i := \si^{-1}(A_i)$. 

	Let $E_i'$ be the set of limit points of $E_i$. 
	Then $[0,1] = \bigcup_{i=1}^k E_i'$. 
	Let $i_1 \in \{1,\ldots,k\}$ be such that $t_0 := 0 \in E_{i_1}'$.   
	If $t_1 < 1$, then there exists $i_2 \in \{1,\ldots,k\}\setminus \{i_1\}$ such that $t_1 \in E_{i_2}'$ 
	and $t_2 := \sup E_{i_2}' > t_1$, by \Cref{lem:sequenceofcells}. 
	Moreover, $[t_1,b] = \bigcup_{j\ne i_1} E_{j}' \cap [t_1 \cap b]$. 
	If $t_2 < b$ we may
	apply \Cref{lem:sequenceofcells} again and find $i_3 \in \{1,\ldots,k\} \setminus \{i_1,i_2\}$ 
	such that $t_2 \in E_{i_3}'$ 
	and $t_3 := \sup E_{i_3}' > t_2$.
	This procedure ends after finitely many steps and gives a finite partition $0 = t_0 < t_1 < \cdots < t_{h-1} < t_h = 1$ of $[0,1]$. The points 
	$y=z_0,z_1,\ldots,z_h=z$, where $z_j = \si(t_j)$, all lie in $U \cap \interior(X)$. 
	Let $\ep >0$ be sufficiently small such that the balls $B_j := B(z_j,\ep)$ are all contained in $U \cap \interior(X)$. 
	For all $j = 1,2,\ldots,h-1$, 
	there exist $z_j^- \in B_j \cap A_{i_j}$ and $z_j^+ \in B_j \cap A_{i_{j+1}}$, by construction. 
	Additionally, there exist $z_0^+ \in B_0 \cap A_{i_1}$ and $z_h^- \in B_h \cap A_{i_h}$.

	Since the cells are quasiconvex, 
	for all $j = 1,2,\ldots,h$, there exist rectifiable paths $\ga_j \in A_{i_j}$ joining $z_{j-1}^+$ and $z_j^-$ such 
	that 
	\[
		\ell(\ga_j) \le C_j \, |z_{j-1}^+ - z_j^-|,
	\]	
	where the constant $C_j$ depends only on $A_{i_j}$. 
	Joining the paths $\ga_j$ with the line segments $[z_0,z_0^+]$, $[z_j^-,z_j^+]$, for $j = 1,\ldots,h-1$, and 
	$[z_h^-,z_h]$, we obtain a rectifiable path $\ga$ in $\interior(X)$ 
	which connects $y$ and $z$ and has length
	\[
		\ell(\ga) \le C  \on{diam}(U),
	\]
	for a constant $C$ which depends only on the $C_j$ and the number of cells $k$, 
	since all points $z_j$, $z_j^-$, $z_j^+$ lie in $U$.
\end{proof}

\subsection{Proof of \texorpdfstring{\Cref{main:4}}{Theorem 1.14}}

	The inclusion $\cC^\infty(X) \subseteq \cA^\infty(X)$ is clear.

	Let $f \in \cA^\infty(X)$. Then $f$ is smooth in $\interior (X)$, 
	by \Cref{result:1}. 
	We must show that $f \in \cC^\infty(X)$. This is a local problem, so we may assume 
	without loss of generality that $X$ is compact (by intersecting with a suitable ball).
	By \Cref{lem:converse} and \Cref{thm:subanreg}, it suffices to show that 
	$f$ satisfies \Cref{lem:converse}(3).

	Fix $x \in \p X$.
	By \Cref{thm:PP}, there is a polynomial curve $h_x : \R \to \R^d$ of degree at most $n$ with 
	the following properties:
	\begin{enumerate}
			\item $h_x((0,1]) \subseteq \interior (X)$ and $h_x(0)=x$,
			\item $\on{dist}(h_x(t),\R^d \setminus X) \ge Mt^m$ for all $t \in (0,1]$, 
	\end{enumerate}
	where $n,M,m$ are independent of $x$ and $t$.
	Then there is a positive integer $k=k(x)$ such that $h_x(t) - x = t^k \tilde h_x(t)$, where $\tilde h_x(0) \ne 0$. 
	Set $v_1 := \tilde h_x(0)/|\tilde h_x(0)| \in S^{d-1}$. 
	Choose $d-1$ directions $v_2,\ldots, v_d \in S^{d-1}$ such that $v_1,v_2,\ldots, v_d$ are linearly independent and define 
	\[
		\Ps_{x,v}(t_1,t_2,\ldots,t_d) := h_x(t_1) + t_2 v_2 + \cdots + t_d v_d  
	\] 
	for $t = (t_1,\ldots,t_d)$ in the set
	\[
		Y :=\big\{(t_1,\ldots,t_d) \in \R^d : t_1 \in (0,\de),\, |t_j| < C t_1^m \text{ for } 2\le j \le d\big\}.
	\]
	If $C:= M/(2(d-1))$ and $\de>0$ is chosen small enough, 
	then $\Ps_{x,v}$ is a diffeomorphism of $Y$ onto the open subset $H_{x,v} := \Ps_{x,v}(Y)$ 
	of $\interior (X)$ and it extends to a homeomorphism between $Y\cup\{0\}$ and $H_{x,v} \cup \{x\}$; indeed, by (2),
	\begin{align*}
		\on{dist}(\Ps_{x,v}(t),\R^d \setminus X) &\ge \on{dist}(h_x(t_1),\R^d \setminus X) - |t_2| - \cdots - |t_d|
		\\
		& > M t_1^m - (d-1) C t_1^m = \frac{M}{2} t_1^m >0,
	\end{align*}
	for $t \in Y$. Since $f$ is smooth in $\interior (X)$, we have 
	\begin{equation} \label{eq:suban1}
	 	\p_{t_2}^{\al_2} \cdots \p_{t_d}^{\al_d} (f \o \Ps_{x,v})(t) =  d_{v_2}^{\al_2} \cdots d_{v_d}^{\al_d} f(\Ps_{x,v}(t)), 
	 	\quad \text{ for all } t \in Y,\, \al_j\ge0. 
	 \end{equation} 
	The left-hand side of \eqref{eq:suban1} extends continuously to $t=0$, since $f \o \Ps_{x,v} \in \cA^\infty(\ol Y)$ 
	and $\cA^\infty(\ol Y)  = \cC^\infty(\ol Y)$, by \Cref{main:1}, as $\ol Y$ is a H\"older set. 
	Since $\Ps_{x,v}$ is a homeomorphism $Y\cup \{0\} \to H_{x,v} \cup \{x\}$, 
	we may conclude that the directional derivatives 
	$d_{v_2}^{\al_2} \cdots d_{v_d}^{\al_d} f$, $\al_j\ge0$, extend continuously from $H_{x,v}$ to $x$. 
	
	If we perturb the directions $v_2,\ldots, v_d$ a little such that $v_1,v_2,\ldots,v_d$ remain
	linearly independent and take the intersection $H_x$ of the corresponding sets $H_{x,v}$, 
	then $H_x$ still is an open subset of $\interior (X)$ with $h_x(t) \in H_x$ for small $t>0$ and $x \in \ol H_x$. 
	Then $d_{w_2}^{\al_2} \cdots d_{w_d}^{\al_d} f$, $\al_j\ge0$, extend continuously from $H_{x}$ to $x$
	for all $w_2,\ldots,w_d$ near $v_2,\ldots, v_d$.
	By \Cref{lem:converse}, 
	we infer that the Fr\'echet derivatives $f^{(p)}$ of all orders of $f$ extend continuously from 
	$H_x$ to $x$.

	Thus for all $x \in \p X$ and $p \in \N$, 
	we have a candidate for the Fr\'echet derivative $f^{(p)}(x)$ of $f$ at $x$ and an open set 
	$H_x \subseteq \interior (X)$ on which $f^{(p)}(y)$ tends to this candidate as $y \to x$. 
	It remains to prove that
	the thus defined extension of $f^{(p)}$ to $X$ is continuous on $X$. 
	First we show that it is bounded.

	\begin{claim} \label{claim1}
		For all $p \in \N$, $f^{(p)}$ is bounded on $X$.
	\end{claim}

	Let $p \in \N$ be fixed. It suffices to show that 
	$f^{(p)}$ is bounded on $\interior (X)$ (since $X$ is fat).
	For contradiction suppose that there is a sequence $(x_\ell)$ in $\interior (X)$ such that 
	$\|f^{(p)}(x_\ell)\|_{L_p} \to \infty$. 
	Since $X$ is compact, we may assume that $x_\ell \to x$. Then $x \in \p X$, 
	since we already know that $f$ is smooth on $\interior(X)$.

	By \Cref{cor:PP}, there is a finite number of real analytic maps $\vh_j : \R^d \times \R \to \R^d$ such that 
	\begin{gather*}
		\vh_j(I^d \times (0,1]) \subseteq \interior (X) \quad \text{ for all }j,
		\\ 
		\bigcup_j \vh_j(I^d \times \{0\}) = X,
	\end{gather*}
	where $I^d := [-1,1]^d$. 
	After passing to a subsequence we may assume that $x_\ell \in \vh_{j_0}(I^d \times \{0\})$ for all $\ell$ 
	and some $j_0$. 
	Choose $y_\ell \in I^d$ such that $\vh_{j_0}(y_\ell,0) = x_\ell$. Since $I^d$ is compact, 
	after passing to a subsequence we may assume that $y_\ell \to y$ and in turn that this convergence is 
	faster than any polynomial. 
	The infinite polygon through the points $y_\ell$ and $y$ can be parameterized by a smooth curve $c : \R \to I^d$ 
	such that $c(1/\ell) = y_\ell$ and $c(0) = y$ (cf.\ \cite[Lemma 2.8]{KM97}). 
	Then $s \mapsto \vh_{j_0}(c(s),0)$ is a smooth curve in $X$ through the points $x_\ell$ and $x$. 

	Since $\vh_{j_0}$ is real analytic, for small $t_1$ we have $\vh_{j_0}(y,t_1) = x + t_1^k \tilde \vh_{j_0}(y,t_1)$ 
	for some positive integer $k$ and a real analytic map $\tilde \vh_{j_0}$ with $\tilde \vh_{j_0}(y,0) \ne 0$. 
	Then $\tilde \vh_{j_0}(z,t_1) \ne 0$ for $(z,t_1)$ in a neighborhood of $(y,0)$.
	Thus,  
	\[
		v_1(z,t_1) := 
		\begin{cases}
			\frac{\p_{t_1}\vh_{j_0}(z,t_1)}{|\p_{t_1} \vh_{j_0}(z,t_1)|} 
		 = \frac{k  \tilde \vh_{j_0}(z,t_1) + t_1 \tilde \vh_{j_0}'(z,t_1) }
		 {|k  \tilde \vh_{j_0}(z,t_1) + t_1 \tilde \vh_{j_0}'(z,t_1)|}	
		 & \text{ if } t_1>0,
		 \\
		 \frac{  \tilde \vh_{j_0}(z,0) }{|  \tilde \vh_{j_0}(z,0) | }
		 & \text{ if } t_1=0,
		\end{cases}
	\] 
	is continuous in $(z,t_1)$, where $t_1\ge 0$, near $(y,0)$.   
	It follows that we can find an open set of directions $v \in S^{d-1}$ such that 
	$v_1(c(s),0)$ and $v$ are linearly independent for $s$ near $0$.
	For such $v$,
	\[
		(s,t_1,t_2) \to f\big(\vh_{j_0}(c(s),t_1) + t_2 v\big)
	\] 
	is smooth for small $s \in \R$, $t_1 \ge 0$, and $|t_2| \le C t_1^m$, by the arguments in 
	\Cref{sec:PP} and the considerations in the first part of the proof.
	But this implies that $d_v^p f(x_\ell)$ is bounded for all such $v$, and hence $f^{(p)}(x_\ell)$ is bounded, 
	a contradiction. \Cref{claim1} is proved.

	\begin{claim} \label{claim2}
		The Fr\'echet derivatives $f^{(p)}$, $p \in \N$, are continuous on $X$.
	\end{claim}

	Let $x \in \p X$ and suppose that $(x_n)$ and $(y_n)$ are two sequences in 
	$\interior (X)$ both converging to $x$. 
	By \Cref{lem:quasiconvex}, for each $\ep>0$ there exists $n_0 \in \N$ such that for all $n \ge n_0$ 
	the points $x_n$ and $y_n$ can be joined by a rectifiable path $\ga_n$ in $\interior(X)$ with 
	length $\ell(\ga_n) \le \ep$.
	Since $f$ is smooth in $\interior (X)$, we may apply 
	the fundamental theorem of calculus and \Cref{claim1} to conclude 
	\[
	\|f^{(p)}(x_n) - f^{(p)}(y_n)\|_{L_p} \le \Big(\sup_{z \in \ga_n} \|f^{(p+1)}(z)\|_{L_{p+1}}\Big) \, \ell(\ga_n) \to 0 
	\quad \text{ as } n \to \infty. 
	\]
	If we assume that the sequence $(x_n)$ lies in $H_x$, we obtain that $f^{(p)}(y) \to f^{(p)}(x)$ for all 
	$\interior (X) \ni y \to x$. 
	Finally, suppose that $\p X \ni x_n \to x$. 
	Choose $y_n \in H_{x_n} \cap B(x_n,1/n)$.
	Then 
	\[
		\|f^{(p)}(x) - f^{(p)}(x_n)\|_{L_p} \le \|f^{(p)}(x) - f^{(p)}(y_n)\|_{L_p} +
		 \|f^{(p)}(x_n) - f^{(p)}(y_n)\|_{L_p} \to 0 
	\]
	as $n \to \infty$. 
	This proves \Cref{claim2} and hence the theorem. \qed

\section{The Bochnak--Siciak theorem on tame closed sets} \label{sec:BS}

In this section we prove \Cref{main:6}.
The strategy for the proof is the following.
Since $f \in \cC^\infty(X)$, we can associate with every $x \in X$ the formal Taylor series $F_x$ 
of $f$ at $x$. Using a result of Eakin and Harris \cite{EakinHarris77} and Gabrielov \cite{Gabrielov73} 
we show that each $F_x$ is convergent and coincides 
with $f$ on their common domain. To prove that all $F_x$ glue together to give a global holomorphic extension we will 
use the following lemma. 

\begin{lemma} \label{lem:shrink} 
	Let $X \subseteq \R^d$ be closed and let $U \subseteq \R^d$ be open with $U \cap X \ne \emptyset$.
	Then there is an open subset $U_0$ of $U$ with $U_0 \cap X = U \cap X$ and such that 
	for all $x \in U_0$ and all $a \in A_x := \{a \in X : |a-x| = \on{dist}(x,X)\}$ we have 
	$[x,a] \subseteq U_0$.
\end{lemma}

\begin{proof} 
	Set $U_0 := \big\{x \in U : [x,a] \subseteq U \text{ for all } a \in A_x\big\}$.
	Then, for all $x \in U_0$ and all $a \in A_x$, we have $[x,a] \subseteq U_0$. 
	For, let $y \in [x,a]$. If $y = x$ there is nothing to prove.
	Otherwise $A_y = \{a\}$ and $[y,a] \subseteq [x,a] \subseteq U$ (as in the figure).

	\begin{figure}[ht]
		\includegraphics[scale=0.3]{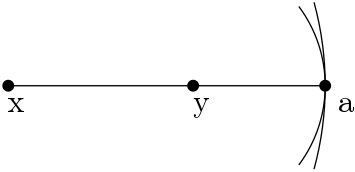}
	\end{figure}

	Clearly, $U_0 \cap X = U \cap X$. 
	It remains to show that $U_0$ is open. 
	To this end we first observe that, if $x_n \to x$ and $A_{x_n} \ni a_n \to a$, then $a \in A_x$.
	This follows from letting $n \to \infty$ in
	$|x_n - a_n| = \on{dist}(x_n,X)$,  
	since $X$ is closed.

	If $U_0$ is not open, then there exists a sequence $x_n \to x$, where $x_n \not \in U_0$ 
	and $x \in U_0$. 
	So, for all $n$, 
	there is $a_n \in A_{x_n}$ and $y_n \in [x_n,a_n] \setminus U$. 
	After passing to a subsequence, we may assume that $a_n \to a \in A_x$, by the observation above, 
	and in turn that $y_n \to y \in [x,a]$. 
	Since $x \in U_0$ we have $y \in U$, a contradiction.
\end{proof}

\begin{proof}[Proof of \Cref{main:6}]
	Suppose that $X \subseteq \R^d$ is a fat closed subanalytic set.
	There exists an analytic manifold $M$ and a proper analytic map $\vh : M \to \R^d$ 
	such that $X = \vh(M)$, by the uniformization theorem, see e.g.\ \cite{BM88}.
	Then $f \o \vh$ is $\cC^\infty$ and real analytic on real analytic curves in $M$. 
	By the Bochnak--Siciak theorem (\Cref{result:4}), $f \o \vh$ is analytic on $M$. 
	For each $x \in X$ there is $y \in \vh^{-1}(x)$ such that $\vh$ has generic rank $d$ at $y$.  
	By a result of Eakin and Harris \cite{EakinHarris77} (proved earlier by Gabrielov \cite{Gabrielov73}), 
	the homomorphism $\vh^*$ of formal power series rings given by formal composition with $\vh$ at $y$ 
	is strongly injective, that is, the formal Taylor series $F_x$ of $f$ at $x$ converges. 
	It represents a holomorphic function $F_x$ in a neighborhood $U_x$ of $x$ in $\C^d$ which coincides with 
	the real analytic function
	$f|_{\interior(X)}$ on $\interior(X) \cap U_x$.  

	It remains to show that the $F_x$ piece together to give a 
	global holomorphic extension of $f$ to a neighborhood of $X$ in 
	$\C^d$.
	We may assume that 
	\begin{equation} \label{eq:productstructure}
		U_x = U_x^\R \times i(-r_x,r_x)^d,	
	\end{equation} 
	where $U_x^\R \subseteq \R^d$. 
	We use \Cref{lem:shrink} to replace each $U_x^\R$ by the connected component of $(U_x^\R)_0$ 
	which contains $x$ (and leave the part of $U_x$ in $i\R^d$ unchanged). 
	Thus we may assume that the cover $\{U_x^\R\}$ of $X$ 
	has the property that for each $z \in U_x^\R$ all segments $[z,a]$, $a \in A_z$, belong to $U_x^\R$. 
 	By \eqref{eq:productstructure}, each $U_x$ has the property that 
	for $z+iw \in U_x$ also $z+itw \in U_x$ for all $t \in [0,1]$.

	Let $V$ be a connected component of $U_x \cap U_y$.
	It follows that, if $z + iw \in V$, then $z \in V^\R := V \cap \R^d$, and 
	$V^\R$ is a connected component of $U_x^\R \cap U_y^\R$. 
	Moreover, $[z,a] \subseteq V^\R$ for all $a \in A_z \subseteq X$. 
	Since $X = \ol{\interior (X)}$, the intersection $V^\R  \cap \interior (X)$ is non-empty and 
	on this set 
	the holomorphic extensions $F_x$ and $F_y$ coincide with $f$. 
	By the identity theorem, $F_x$ and $F_y$ coincide 
	on $V$. Since the component $V$ of $U_x \cap U_y$ was arbitrary, $F_x$ and $F_y$ coincide on $U_x \cap U_y$.
\end{proof}

\begin{proof}[Proof of \Cref{main:7}]
	The assumption for $X$ clearly implies that $X=\ol {\interior(X)}$.
	For each boundary point $z \in \p X$ there is a holomorphic function $F_z$ defined in a neighborhood $U_z$ of $z$ 
	in $\C^d$ which coincides with $f$ on $U_z \cap \interior(X)$; this follows from \Cref{main:6} applied to the 
	subanalytic set $X_z$. Using \Cref{lem:shrink} as in the proof of \Cref{main:6}, 
	one easily concludes the assertion.
\end{proof}

\section{Arc-\texorpdfstring{$\cC^M$}{CM} functions on Lipschitz sets} \label{sec:mainB}

In this section we prove \Cref{main:2}:
\emph{All $X \in \sH^1(\R^d)$ satisfy 
\[
\cC^M(X) \subseteq \cA^M(X) \subseteq \cC^{M^{(2)}}(X), 
\]
for any non-quasianalytic weight sequence $M = (M_k)$.}

It can be expected that a similar statement holds for $X \in  \sH^\al(\R^d)$, where $\al<1$,
with a \emph{larger} weight sequence $N = N(\al,M)$ instead of $M^{(2)}$.  
We do not pursue this question any further for $\al$-sets, but results of this type for subanalytic sets are 
presented in \Cref{ssec:ultra}. 

\subsection{Reduction to an open set of directions}

Let $f : \R^2 \to \R$ be smooth. 
The mixed partial derivatives of order $k$ of $f$ at any point $x \in \R^2$ can be computed from directional derivatives 
of order $k$ of $f$ at $x$ 
by means of the identity
\begin{equation} \label{system}
	d_v^k f(x) = \sum_{j=0}^k \binom{k}{j} v_1^j v_2^{k-j} \p_1^j \p_2^{k-j} f(x), \quad v = (v_1,v_2) \in \R^2. 
\end{equation}
The next lemma guarantees that the constants which appear in the process of solving these linear equations 
grow at most exponentially in $k$ and hence the class $\cC^M$ is preserved; a similar lemma was proved in \cite{Neelon99}.

\begin{lemma} \label{Neelon}
Let $-1 \le t_0 < t_1 < \cdots < t_k \le 1$ be equidistant points such that $t_k -t_0 = a$. 
If $x_0,x_1,\ldots,x_k$ is a solution of the linear system of equations
\begin{equation} \label{eq:linear}
	\sum_{j=0}^k \binom{k}{j} t_i^j x_j = y_i, \quad  i = 0,1,\ldots,k,
\end{equation}
then we have 
\begin{equation} \label{change}
		\max_j |x_j| \le \Big(\frac{16e^2}{a}\Big)^k \max_m |y_m|.	
\end{equation}
\end{lemma}

\begin{proof}
	Let $P(t) = a_0 + a_1 t + \cdots + a_k t^k$ be the polynomial with coefficients $a_j = \binom{k}{j} x_j$. 
	Then the system \eqref{eq:linear} reads
	\[
		P(t_i) = y_i, \quad i = 0,1,\ldots,k.  
	\]
	By Lagrange's interpolation formula 
  (e.g.\ \cite[(1.2.5)]{RS02}), 
  \[
    P(t) = \sum_{i=0}^k y_i \prod_{\substack{j=0\\ j\ne i}}^k \frac{t-t_j}{t_i-t_j},
  \]
  and therefore
  \[
    a_m = (-1)^{k-m} \sum_{i=0}^k y_i  \si^i_{k-m} \prod_{\substack{j=0\\ j\ne i}}^k \frac{1}{t_i-t_j} ,
  \]
  where $\si^i_j$ is the $j$th elementary symmetric polynomial in $(t_\ell)_{\ell \ne i}$.
  We have   
	\[
		|t_i-t_j| = \frac{a|i-j|}{k}, \quad 	|t_j| = \Big|t_0 + a \frac{j}{k}\Big| \le 2\frac{k+j}{k}, 
	\]
	and hence, using $e^{-k}k^k \le k! \le k^k$, 
	\begin{align*}
		\prod_{\substack{j=0\\ j\ne i}}^k \frac{1}{|t_i-t_j|}  = \frac{k^k}{a^k i! (k-i)!} \le \Big(\frac{2e}{a}\Big)^k,   
	\end{align*}
	and
	\begin{align*}
		|\si^i_{k-m}| 
		&\le \binom{k}{m} \Big(\frac{2}{k}\Big)^{k-m} 
		\frac{(2k)!}{(k+m)!} \le \binom{k}{m} \Big(\frac{2}{k}\Big)^{k-m} 4^k (k-m)! \le \binom{k}{m} 8^k.
	\end{align*}
	It follows that 
	\begin{equation*}
		|x_m| \le \Big(\frac{16e}{a}\Big)^k \sum_{i=0}^k |y_i| \le  \Big(\frac{16e^2}{a}\Big)^k \max_i |y_i|,  
	\end{equation*}
	that is \eqref{change}.
\end{proof}

\begin{proposition} \label{conesuffices}
		Let $f : \R^d \to \R$ be smooth. Let $K \subseteq \R^d$ be compact and let $M=(M_k)$ be a positive 
		sequence.
		The following assertions are equivalent:
		\begin{enumerate}
		\item $\exists C,\rh>0 \A k \in \N \A x \in K \A v \in S^{d-1} :  |d_v^k f(x)| \le C \rh^k k!\, M_k$.  
		\item There exist $v_0 \in S^{d-1}$ and $r>0$ such that \newline
		$\exists C,\rh>0 \A k \in \N \A x \in K \A v \in B(v_0,r) \cap S^{d-1} :  |d_v^k f(x)| \le C \rh^k k!\, M_k.$ 
		\item $\exists C,\rh>0 \A x \in K \A \al \in \N^d :  |\p^\al f(x)| \le C \rh^{|\al|} |\al|!\, M_{|\al|}.$
	\end{enumerate}
	The constants $C$, $\rh$ may differ from item to item, but they change in a uniform way which depends only on 
	$r$.
\end{proposition}

\begin{proof}
	Let us first consider the case $d=2$. In this case $B:=B(v_0,r) \cap S^{d-1}$ is an open arc $I \subseteq S^1$; 
	let $\ell(I)$ denote the 
	length 
	of $I$.

	(1) $\Rightarrow$ (2) is trivial and (3) $\Rightarrow$ (1) follows easily from \eqref{system}. 

	(2) $\Rightarrow$ (3)
	By a linear change of coordinates, we may assume that the arc $I$ is symmetric about the $y$-axis and 
	by shrinking $I$, we may also assume that its projection to the $y$-axis is contained in 
	$\{(0,y) : 1/2 \le y \le 1\}$ and that the estimates in (2) hold also at the endpoints of $I$. 
	Let $(-a/2,a/2)$ be the projection of $I$ to the $x$-axis and let $-a/2 =t_0 < t_1 <\cdots <t_k =a/2$ 
	be an equidistant partition. 
	Apply \Cref{Neelon} to the system \eqref{system} with the $k+1$ directions $v_i =(t_i,s_i)$, $i=0,\dots,k$, in $I$;	
	then $1/2 \le s_i \le 1$.
	The statement about the uniform change of the constants follows from \eqref{change}. 

	Now we consider the general case.

	(1) $\Leftrightarrow$ (2)
	The statement follows by applying the 2-dimensional analogue 
	to every affine 2-plane $\pi$ containing the affine line $x + \R v_0$. The change of the constants $C$, $\rh$ 
	depends only on the length of the arcs defined by the intersection $\pi \cap B$ which is independent of $\pi$. 

	(1) $\Leftrightarrow$ (3) By the polarization formula \cite[Lemma 7.13(1)]{KM97}, we have
	\begin{equation*}
		\sup_{|v| \le 1} |d^k_v f(x)| \le \|d^k f(x)\|_{L_k} \le (2e)^k  \sup_{|v| \le 1} |d^k_v f(x)|
	\end{equation*}
	which entails the assertion.
\end{proof}

\subsection{Proof of \texorpdfstring{\Cref{main:2}}{Theorem 5.1}}
	Let $M=(M_k)$ be a non-quasianalytic weight sequence.
	Let $X \in \sH^1(\R^d)$.
	The inclusion $\cC^M(X) \subseteq \cA^M(X)$ is an easy consequence of Fa\`a di Bruno's formula and log-convexity of $M$
	(cf.\ \cite[Proposition 3.1]{RainerSchindl12}). 

	Let us prove $\cA^M(X) \subseteq \cC^{M^{(2)}}(X)$. 
	A function $f \in \cA^M(X)$ belongs to $\cC^\infty(X)$, by \Cref{main:1}. 
	Suppose for contradiction that $f \not\in \cC^{M^{(2)}}(X)$. 
	Then there is $a \in X$ such that for all $\de,C,\rh>0$ there exist 
	$x \in X \cap B(a,\de)$, $v \in S^{d-1}$, and $k \in \N$ with 
	\begin{equation} \label{eq:contra}
		|d^k_v f(x)| > C \rh^k k!\, M^{(2)}_k. 
	\end{equation}
	We may assume that $a \in \p X$ (if $a \in \interior (X)$ then the arguments in the proof of \cite[Theorem 3.9]{KMRc} 
	lead to a contradiction). 
	Since $X \in \sH^1(\R^d)$, we may suppose that
	there exist $\ep >0$ and a truncated open cone $\Ga = \Ga^1_d(r,h)$ such that 
	\begin{equation} \label{eq:cone}
		\text{for all $y \in X  \cap B(a,\ep)$ we have $y + \Ga \subseteq \interior (X)$.}	
	\end{equation}
	By rescaling, we may assume that $r=h=1$.
	Set $C(y,r) := y + \Ga_d^1(r,r)$ for $0< r \le 1$.  
	There is a universal constant $c >0$  
	such that $C(y_1,r_1) \cap C(y_2,r_2) \ne \emptyset$ if $|y_1 - y_{2}| < c \min\{r_1,r_2\}$. 

	Let $\la_k \searrow 0$ be the sequence associated with the sequence $M_k$, by \Cref{curvelemma}.
	By \Cref{conesuffices} and \eqref{eq:contra} (using $\de := c\la_{n+1}/3$, $C := \la_n^{-1}$, $\rh:= \la_n^{-3}$),
	there exist sequences $x_n \in  X \cap  B(a,c\la_{n+1}/3)$, $v_n \in S^{d-1} \cap \R_+ \Ga$, 
	$k_n \in \N$
	such that 
	\begin{equation} \label{eq:contradiction}
		|d_{v_n}^{k_n} f(x_n)| \ge \la_n^{-3k_n -1} k_n!\, M^{(2)}_{k_n}  \quad \text{ for all } n.
	\end{equation}
	Let us set $C_n := C(x_n,\la_n)$.
	Since $|x_n -x_{n+1}| < c \la_{n+1}$,
	there is a sequence $u_n$ such that $u_{n+1} \in C_n \cap C_{n+1}$ for all $n$. 
	Evidently, $x_n$ and $u_n$ are both $1/\la_n$-converging to $a$. 
	We may assume that for all $n\ge n_0$ we have $C_n \subseteq \interior (X)$, by \eqref{eq:cone}.  

	Without loss of generality assume that $a =0$. 
	Let $c_n(t) = x_n + t^2 \la_n v_n$. Let $T_n$ and $t_n$ be chosen as in \eqref{eq:Tt}, 
	and let $\vh$ be the function used in the proof of \Cref{curvelemma}.
	Define 
	\[
		c(t) =  \vh\Big(\frac{t-t_n}{T_n}\Big) c_n(t-t_n) + \Big(1- \vh\Big(\frac{t-t_n}{T_n}\Big) \Big) 
		\Big(u_n \mathbb {1}_{(-\infty, t_n]}(t) + u_{n+1}\mathbb {1}_{[t_n,+\infty)}(t)\Big)
	\]
	for $t \in [t_n -T_n, t_n+T_n]$ (here $\mathbb{1}_A$ denotes the characteristic function of the set $A$); 
	note that $t_n+T_n = t_{n+1} - T_{n+1}$. 

	\begin{figure}[ht]
		\includegraphics[scale=0.27]{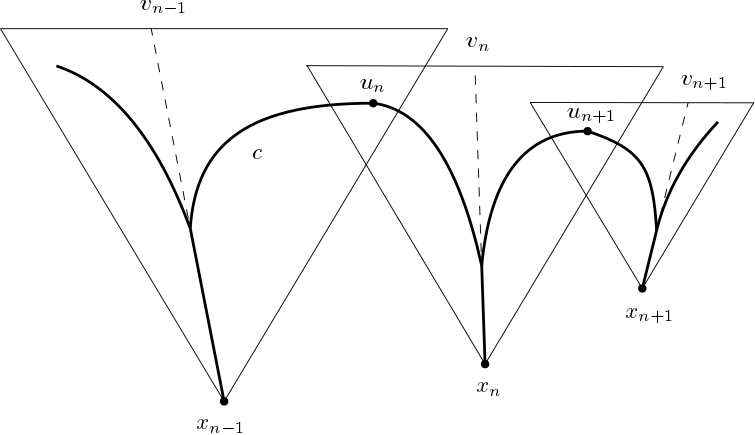}
	\end{figure}

	Extend $c$ by $c=0$ on $[t_\infty,\infty)$.
	Then $c$ is $\cC^\infty$ on $[t_{n_0} - T_{n_0},+\infty) \setminus \{t_\infty\}$ and $c(t_n-T_n) = u_{n}$ and $c(t_n+T_n) = u_{n+1}$. 
	By construction, $c(t) \in C_n$ if  $t \in [t_n -T_n, t_n+T_n]$ and thus $c$ lies in $X$. 
	Since the curves $c_n$ as well as $u_n$ satisfy \eqref{eq:lamdaconverge}, the proof of \Cref{curvelemma} implies 
	that $c$ is a $\cC^M$-curve.

	Then, since $f \in \cA^\infty(X)$, for all $k$, 
	\[
		(f\o c)^{(2k)}(t_n) = \frac{(2k)!}{k!} \la_n^k d_{ v_n}^k f(x_n).
	\]
	Using \eqref{eq:contradiction}, 
	we may conclude 
	\begin{align*}
		\Big( \frac{|(f\o c)^{(2k_n)}(t_n)|}{(2 k_n)!\, M_{2 k_n}} \Big)^\frac{1}{2 k_n + 1} 
		&= \Big( \frac{\la_n^{k_n} |d_{v_n}^{k_n} f(x_n)|}{ k_n!\, M_{2 k_n}} \Big)^\frac{1}{2 k_n + 1}
		\ge \frac{1}{\la_n}  \to \infty,
	\end{align*}
	as $n \to \infty$,
	contradicting the assumption $f \in \cA^M(X)$. \qed

\section{Arc-\texorpdfstring{$\cC^M$}{CM} functions on subanalytic sets}
 \label{ssec:ultra}

Let $M=(M_k)$ be a non-quasianalytic weight sequence.
Let $X$ be a simple fat closed subanalytic set.
We will see in this section that $\cA^M(X) \subseteq \cC^N(X)$ for some other  
non-quasianalytic weight sequence $N$ which depends only on $M$ and $X$ (in an explicit way).

\subsection{Rectilinearization}

We start with some simple observations.
For arbitrary sets $Y \subseteq \R^e$, $X \subseteq \R^d$ 
we denote by $\cC^\infty(Y,X)$ the class of mappings $\vh : Y \to X$ such that 
$\vh_i \in \cC^\infty(Y)$ for all components $\vh_i = \on{pr}_i \o \vh$. 
Similarly, for $\cC^M(Y,X)$ and $\cC^\om(Y,X)$. 

\begin{lemma} \label{lem:pullback}
	Let $X \subseteq \R^d$ and $Y \subseteq \R^e$. 
	We have:
	\begin{enumerate}
		\item If $\vh \in \cC^\infty(Y,X)$ and $\cA^\infty(Y) = \cC^\infty(Y)$, 
		then $\vh^* \cA^\infty(X)\subseteq \cC^\infty(Y)$.
		\item If $\vh \in \cC^\om(Y,X)$ and $\cA^\om(Y) = \cC^\om(Y)$, then $\vh^* \cA^\om(X)\subseteq \cC^\om(Y)$.  
		\item If $\vh \in \cC^M(Y,X)$ and $\cA^M(Y) \subseteq  \cC^{N}(Y)$, then $\vh^* \cA^M(X)\subseteq \cC^{N}(Y)$.
	\end{enumerate}
\end{lemma}

\begin{proof}
	We prove (1); (2) and (3) work similarly. 
	Let $f \in \cA^\infty(X)$. 
	Assume that $f \o \vh \not \in \cC^\infty(Y)$. Since $\cC^\infty(Y) =  \cA^\infty(Y)$, 
	there exists 
	$c \in \cC^\infty(\R,Y)$ such that $f \o \vh \o c \not \in \cC^\infty(\R,\R)$. 
	But $\vh \o c$ is a $\cC^\infty$-curve in $X$, contradicting $f \in \cA^\infty(X)$. 
\end{proof}

Combining this lemma with the rectilinearization of subanalytic sets (see \Cref{rectilinearization})
we conclude the following.

\begin{theorem} \label{main2}
	Let $M=(M_k)$ be a non-quasianalytic weight sequence.
	Let $X \subseteq \R^d$ be a fat closed subanalytic set.
	There is a locally finite collection of real analytic mappings $\vh_\al : U_\al \to \R^d$, where the $U_\al$ 
	are open sets in $\R^d$, 
	such that, for all $\al$, 
	\begin{align}
		\vh_\al^* \cA^\infty(X) &\subseteq \cC^\infty(\vh_\al^{-1}(X)), 
		\\
		\vh_\al^* \cA^\om(X) &\subseteq \cC^\om(\vh_\al^{-1}(X)),
		\\
		\vh_\al^* \cA^M(X) &\subseteq \cC^{M^{(2)}}(\vh_\al^{-1}(X)).
	\end{align}
\end{theorem}

\begin{proof}
	We use \Cref{rectilinearization}.
	Since $X = \overline {\interior (X)}$,  
we may assume that, for the quadrants $Q(I_0,I_-,I_+)$ whose union is $\vh_\al^{-1}(X)$, 
we have $I_0 = \emptyset$.  
We claim that a union $Y$ of quadrants $Q(\emptyset,I_-,I_+)$ is $\cA^\infty$- and $\cA^\om$-admissible. 
Furthermore, we claim that $Y$ satisfies $\cA^M(Y) \subseteq  \cC^{M^{(2)}}(Y)$.
Then \Cref{lem:pullback} implies the result.

{\it $\cA^\infty$-admissibility.}
By \Cref{main:1}, each  
$Q = Q(\emptyset,I_-,I_+)$ is $\cA^\infty$-admissible. 
	Any two different quadrants $Q_1$, $Q_2$ 
	have non-empty intersection $\pi$ which consists of a coordinate sector of 
	dimension $k \in \{0,\ldots,d-1\}$ (for $k=0$, $\pi = \{0\}$). 
	Suppose that $\pi$ is a coordinate sector of dimension $k$. 
	Let $v \in Q_1 \cup Q_2$ be any vector perpendicular to $\pi$. 
	Then $\si_v :=\pi + \R v$ is a $k+1$ dimensional closed convex set contained in $Q_1 \cup Q_2$.  
	We may conclude that $f|_{\si_v} \in \cC^\infty(\si_v)$. 
	Thus the directional derivatives $d^n_w f$ of $f$ of all orders $n$ at points in $\pi$ 
	with direction $w \in \bigcup_{v \in Q_1 \cup Q_2} \si_v$ exist and are unique.  
	These suffice to compute the partial derivatives of $f$ of all orders at points in $\pi$.
	This proves that $Q_1 \cup Q_2$ is $\cA^\infty$-admissible.
	The general case follows by induction.
	This also proves that we even have $\cA_M^\infty(Y) = \cC^\infty(Y)$.  

{\it $\cA^\om$-admissibility.}
	This follows from \Cref{main:6} and the fact that $Y$ is $\cA^\infty$-admissible.

{\it Finally we show $\cA^M(Y) \subseteq  \cC^{M^{(2)}}(Y)$.} 
	Since we already have $\cA_M^\infty (Y) = \cC^\infty(Y)$, it suffices to check that 
	the estimates \eqref{DCcondition} (for $M^{(2)}$ instead of $M$) hold for all $f \in \cA^M(Y)$ and 
	for each compact $K \subseteq Y$.  
	This is clear, since $f|_Q \in \cA^M(Q) \subseteq \cC^{M^{(2)}}(Q)$, by \Cref{main:2}, 
	for each of the finitely many quadrants $Q$ which make up $Y$.	 
\end{proof}

\subsection{Controlled loss of regularity}

Let $M=(M_k)$ be a weight sequence. Recall that, for positive integers $a$, $M^{(a)}$ denotes the weight sequence 
defined by $M^{(a)}_k := M_{ak}$.

\begin{proposition} \label{prop:loss}
	Let $M=(M_k)$ be a non-quasianalytic weight sequence.
	Let $X \subseteq \R^d$ be a fat compact subanalytic set.
	Then there is a positive integer $a$, independent of $M$, such that
	\begin{equation} 
		\cC^\infty(X) \cap \cA^{M}(X)	 \subseteq  \cC^{M^{(a)}}(X). 	
	\end{equation}
\end{proposition}

\begin{proof}
	Let $\vh_\al$ be the finitely many mappings provided by \Cref{rectilinearization}.  
	We may assume that the Jacobian determinant of each $\vh_\al$ is a monomial times a 
	nowhere vanishing factor.
	Let $f \in \cC^\infty(X) \cap \cA^{M}(X)$. 
	By \Cref{main2},  
	$f \o \vh_\al \in \cC^{M^{(2)}}(Y_\al)$ 
	where $Y_\al$ is a union of quadrants in $\R^d$. 
	By \cite[Theorem 1.4]{BelottoBierstoneChow17}, for each $\al$ there is a positive integer $a_{\al}$ such that $f$ is of class $\cC^{M^{(a_\al)}}$ on 
	$\vh_\al(Y_\al)$. It follows that $f \in \cC^{M^{(a)}}(X)$, where $a = \max_\al a_\al$.  
\end{proof}

For $a \in \R_{>0}$ we may define the weight sequence $M^a$ by $M^a_k := (M_k)^a$. 
If $M = (M_k)$ has moderate growth (see \eqref{mg}) and $a$ is an integer, then there exists $C=C(a)$ such that
\[
	M_k^a \le M_{ak} \le C^k M_k^a \quad \text{ for all } k,  
\] 
i.e., $M^{(a)}$ and $M^a$ define the same Denjoy--Carleman class. Note also that $M^a$ has moderate growth whenever $M$ has.

Assume that for each $a>0$, the weight sequence $M^a$ is non-quasianalytic and define
\[
	 \widehat \cA^{M}(X) := \bigcap_{a>0} \cA^{M^a}(X) 
	\quad \text{ and } \quad \widehat  \cC^{M}(X) := \bigcap_{a>0} \cC^{M^a}(X). 
\]

\begin{theorem} \label{main:5}
    Let $M=(M_k)$ be a weight sequence of moderate growth 
    such that $M^a$ is non-quasianalytic for all $a>0$.
    Let $X \subseteq \R^d$ be a fat closed subanalytic set.
	Then 
	\begin{equation}
		\cC^\infty(X) \cap \widehat  \cA^{M}(X)	 = \widehat \cC^{M}(X). 	
	\end{equation}
	If $X$ is simple, then 	
	\begin{equation}
		\widehat  \cA^{M}(X)	 = \widehat  \cC^{M}(X). 	
	\end{equation}
\end{theorem} 

\begin{proof}
	The inclusion $\widehat  \cC^{M}(X) \subseteq \cC^{\infty}(X) \cap \widehat  \cA^{M}(X)$ is obvious.
	The converse inclusion follows from \Cref{prop:loss}. 	
\end{proof}

\begin{remark}
	Instead of \cite[Theorem 1.4]{BelottoBierstoneChow17} one can also use the results of \cite{ChaumatChollet99}.
\end{remark}

\section{Applications} \label{sec:applications}

\subsection{Solutions of real analytic equations}

\begin{theorem} \label{thm:aneq}
	Let $U \subseteq \R^{d+1}$ be open and let $H : U \to \R$ be a real analytic function (not identically zero). 
	Let $X \subseteq \R^d$ be a closed set such that for all $z \in \p X$ there 
	is a closed fat subanalytic set $X_z$ such that $z \in X_z \subseteq X$; e.g.\ $X$ itself is fat and subanalytic 
	or a H\"older set.	 
	If $f \in \cC^\infty(X)$ satisfies $H(x,f(x)) = 0$ for all $x \in X$, then $f$ extends to 
	a holomorphic function on a neighborhood of $X$ in $\C^d$. 
\end{theorem}

\begin{proof}
	Suppose first that $X \subseteq \R^d$ is fat closed subanalytic.
	As in the proof of \Cref{main:6} there is a proper real analytic map $\vh : M \to \R^d$ 
	with $X = \vh(M)$. 
	Then $(z,y) \mapsto H(\vh(z),y)$ is not identically zero. 
	By the classical version of this theorem, 
	cf.\ \cite{Bochnak70}, \cite{Siciak70}, 
	and \cite{Malgrange67}, we may conclude that $z \mapsto (f\o \vh)(z)$ is real analytic on $M$. 
	The proof of \Cref{main:6} (in \Cref{sec:BS}) then yields the assertion.

	In the general case, fix $z \in \p X$ and a closed fat subanalytic set $X_z$ 
	with $z \in X_z \subseteq X$.
	Then $f|_{X_z} \in \cC^\infty(X_z)$ satisfies $H(x,f(x)) = 0$ for all $x \in X_z$. 
	Thus, by the first part of the proof, $f|_{X_z}$ extends to a holomorphic function on a neighborhood 
	of $X_z$ in $\C^d$.  
	That these local extensions glue to the desired global extension follows from \Cref{lem:shrink} as in 
	the proof of \Cref{main:6}.  
\end{proof}

	We obtain the following corollary for \emph{Nash functions}, i.e., 
	real analytic functions $f : U \to \R$ defined in an open semialgebraic set $U \subseteq \R^d$ 
	which satisfy a non-trivial polynomial equation $P(x,f(x))=0$ for all $x \in U$.

\begin{corollary}
	Let $X \subseteq \R^d$ be a fat closed semialgebraic set 
	and let $f : \interior (X) \to \R$ be a Nash function whose partial derivatives of all orders 
	extend continuously to the boundary of $X$. 
	Then $f$ is the restriction of a Nash function on an open neighborhood of $X$.
\end{corollary}

\begin{proof}
	The extension of $f$ clearly also satisfies the defining polynomial equation.
\end{proof}

\subsection{Composite real analytic functions}

Suppose that $\vh : M \to \R^d$ is a real analytic map, where $M$ is a real analytic manifold.
Assume that $g \in \cC^\infty(\R^d)$ and $f = g \o \vh \in \cC^\om(M)$.
Let $X := \vh(M)$. 
Our results yield a sufficient condition for $g|_X$ to admit a real analytic extension to some 
open neighborhood of $X$.

\begin{corollary} \label{fep}
	Let $\vh : M \to \R^d$ be real analytic and such that:
	\begin{enumerate}
		\item $X := \vh(M)$ is a fat closed subanalytic subset of $\R^d$.
		\item Each $c \in \cC^\om(\R,X)$ admits a lifting $\tilde c \in \cC^\om(\R,M)$, i.e., $c = \vh \o \tilde c$.
	\end{enumerate}
	Then, for each $g \in \cC^\infty(\R^d)$ with $g \o \vh \in \cC^\om(M)$, 
	there exists a holomorphic function $G$ defined in an open neighborhood of $X$ in $\C^d$ such that 
	$g|_X = G|_X$. 
 \end{corollary} 

\begin{proof}
	Follows from \Cref{main:6}.	
\end{proof}

Conditions for the existence of a smooth solution $g$ of the equation $f = g \o \vh$ have been intensively studied; 
see \cite{BierstoneMilman82}, \cite{BierstoneMilmanPawlucki96}, \cite{BierstoneMilman98}. 
 
\begin{remark}
	For instance, the conditions of the corollary are satisfied in the following situation.
	Let $\rh : G \to \on{O}(V)$ be a coregular finite dimensional orthogonal representation of a compact Lie group. 
	Let $\si = (\si_1,\ldots,\si_d)$ be a minimal system of generators of the algebra $\R[V]^G$ of 
	$G$-invariant polynomials.	 
	Schwarz' theorem \cite{Schwarz75} (see also \cite{Mather77}) 
holds that for each $G$-invariant $f \in \cC^\infty(V)$ there exists 
$g \in \cC^\infty(\R^d)$ such that $f = g \o \si$.	
The set $X = \si(V)$ is closed semialgebraic and fat, by the assumption that $\rh$ is coregular, 
	cf.\ \cite{PS85}.
	Real analytic curves in $X$ admit real analytic liftings to $V$, 
by \cite{AKLM00} and \cite[Theorem 4]{ParusinskiRainer14}. 
The corollary implies that every $G$-invariant real analytic function $f$ on $V$ is of the form $f=g \o \si$,
where $g$ is a holomorphic function defined in an open neighborhood of $X$ in $\C^d$. 
A more general result (with a different proof) is due to Luna \cite{Luna76}.
\end{remark}

\subsection{Division of smooth functions and pseudo-immersions}

Statements about smooth functions on open sets can sometimes be reduced to corresponding statements 
for functions of one real variable,  
thanks to Boman's theorem \ref{rem:Boman}. 
This principle extends to $\cA^\infty$-admissible sets. 
We illustrate this using two selected examples.
The first concerns division of smooth functions:

\begin{theorem} \label{thm:division}
	Suppose that $X$ is a H\"older set or a simple fat closed subanalytic subset of $\R^d$. 
	If $f,g : X \to \C$ satisfy
	\begin{enumerate}
		\item $g,fg,f^m \in \cC^\infty(X,\C)$, and
		\item $|f(x)| \le C \, |g(x)|^\al$ for all $x \in X$, 
	\end{enumerate}
	for some $m \in \N_{\ge1}$ and $C,\al>0$, then $f \in \cC^\infty(X,\C)$.
\end{theorem}

\begin{proof}
	This follows from \cite[Theorem 1]{JorisPreissmann90} which is precisely the case $X=\R$, 
	\Cref{main:1}, and \Cref{main:4}.
\end{proof}

In \cite{JorisPreissmann90} this theorem (for $X=\R$) was used to prove that certain maps are 
pseudo-immersions. 
A $\cC^\infty$-map $p : N \to M$ between $\cC^\infty$-manifolds is a \emph{pseudo-immersion} if for each continuous 
map $f : P \to N$, where $P$ is a $\cC^\infty$-manifold, $p \o f \in \cC^\infty$ implies $f \in \cC^\infty$; 
see also \cite{JorisPreissmann87}.  
Pseudo-immersivity of a smooth map is a local property. So it is enough to consider germs of 
smooth maps $p : (\R^n,0) \to (\R^m,0)$. By Boman's theorem \ref{result:1}, the defining universal 
property must be checked only for smooth curves: $p$ is a pseudo-immersion if and only if for each 
(continuous) curve 
$c : \R \to \R^n$ we have the implication $p \o c \in \cC^\infty \implies c \in \cC^\infty$.

The results of \Cref{main:1} and \Cref{main:4} entail the following.    

\begin{theorem} \label{thm:pseudoimmersion}
	Let $p : \R^n \to \R^m$ be a pseudo-immersion. Then the universal property of $p$ extends to 
	maps $f : X \to \R^n$, where $X \subseteq \R^d$ is $\cA^\infty$-admissible, in particular, for $X$
	a H\"older set or 
	a simple fat closed subanalytic subset of $\R^d$.
\end{theorem}

For instance, if $f : X \to \C$ is continuous and $f^2,f^3 \in \cC^\infty(X,\C)$, then $f \in \cC^\infty(X,\C)$.
In addition, by \Cref{thm:aneq}, if at least one of $f^2$ or $f^3$ is real analytic, then also $f$ is real 
analytic.

\section{Complements and examples} \label{sec:complements}

\subsection{\texorpdfstring{$\cC^M$}{CM}-extensions} \label{sec:extension}

Let $X \subseteq \R^d$ be a H\"older set or a fat closed subanalytic set.
By \Cref{lem:converse}, \Cref{prop:alpharegular}, and \Cref{thm:subanreg}, 
any function $f : X \to \R$ which satisfies \Cref{lem:converse}(3) extends to a $\cC^\infty$-function on $\R^d$.
Let us investigate this in the ultradifferentiable case. 
For strongly regular weight sequences $M$ there is a $\cC^M$-version of Whitney's extension theorem \cite{Bruna80}.

\begin{lemma} \label{lem:extension}
	Let $X\subseteq \R^d$ be a fat compact set either in 
	$\sH(\R^d)$ or subanalytic. 
	Suppose there is a positive integer $m$ 
	and a constant $D>0$, such that 
	any two points $x,y \in X$ can be joined by a rectifiable path $\ga$ in $X$ 
	and 
	\begin{equation} \label{m-regular}
		\ell(\ga)^m \le D |x-y|.
	\end{equation}
	Let $M$ be a weight sequence. 
	Then each $f \in \cC^M(X)$ defines a Whitney jet on $X$ of class $\cC^N$  
	where $N_k := M_{mk}$, i.e., there exist constants $C,\rh>0$ such that 
	\begin{gather} \label{eq:DCjetN}
	 	|f^{(\al)}(x)| \le C \rh^{|\al|} |\al|!\, N_{|\al|}, \quad  \al \in \N^d,~x \in X, 
		\\	 \label{eq:remainder}
		|(R^p_x f)^\al(y)| \le C \rh^{p+1} |\al|!\, N_{p+1} |x-y|^{p+1 - |\al|}, \quad p \in \N, ~ |\al| \le p, ~ x,y \in X, 
	\end{gather}
	where 
	\[	
		(R^p_x f)^\al(y) = f^{(\al)}(y) - \sum_{|\be| \le p - |\al|} \frac{f^{(\al + \be)}(x)}{\be!} (y-x)^{\be}.
	\]
\end{lemma}

\begin{proof}
	Let $f \in \cC^M(X)$.
	Now \eqref{eq:DCjetN} is clearly satisfied since we even have 
	\begin{equation} \label{eq:DCjet}
	 	|f^{(\al)}(x)| \le C \rh^{|\al|} |\al|!\, M_{|\al|}, \quad  \al \in \N^d,~x \in X. 
	\end{equation}
	Since $f$ has a smooth extension to $\R^d$, $f$ defines a Whitney jet of class $\cC^\infty$ on $X$. 
	We claim that 
	\begin{equation} \label{eq:DCjet2}
		|(R^p_x f)^\al(y)| \le \frac{(d \ell(\si))^{p+1-|\al|}}{(p+1-|\al|)!}  
		\sup_{\substack{\xi \in \si\\ |\ga| = p+1}} |f^{(\ga)}(\xi)|    
	\end{equation}
	for any rectifiable path $\si$ which joins $x$ and $y$.
	Then, by \eqref{m-regular} and \eqref{eq:DCjet},
	there are constants 
	$C_i,\rh_i>0$ such that 
	\begin{align*}
		|(R^p_x f)^\al(y)| 
		&\le |(R^{m(p+1)-1}_x f)^\al(y)| + 
		\Big|\sum_{p - |\al| < |\be| < m(p+1) - |\al|} \frac{f^{(\al + \be)}(x)}{\be!} (y-x)^{\be}\Big|
		\\
		&\le d^{m(p+1)-|\al|} C \rh^{m(p+1)} |\al|!\, M_{m(p+1)}  \ell(\si)^{m(p+1) - |\al|}  
		\\
		&\quad + 
		 C_1 \rh_1^{m(p+1)} |\al|!\, M_{m(p+1)} |x-y|^{p-|\al|+1}
		 \\
		&\le C_2 \rh_2^{m(p+1)} |\al|!\, M_{m(p+1)}  
				  |x-y|^{p-|\al|+1},
	\end{align*}
	that is \eqref{eq:remainder}.
	To see \eqref{eq:DCjet2} notice that, with $T^p_x f(y) :=  \sum_{|\be| \le p} \frac{f^{(\be)}(x)}{\be!} (y-x)^{\be}$, 
	\[
		(R^p_x f)^\al(y) = f^{(\al)}(y) - T^{p-|\al|}_x f^{(\al)}(y)  = T^{p-|\al|}_y f^{(\al)}(y) - T^{p-|\al|}_x f^{(\al)}(y).
	\]
	By choosing a suitable parameterization, we may assume that
	$\si : [0,1] \to \R^d$ is an absolutely continuous curve from $x$ to $y$ such that $|\si'(t)| = \ell(\si)$ 
	for a.e.\ $t$. Then 
	\begin{align*}
	(R^p_x f)^\al(y) 
	&= \int_0^1 \p_t (T^{p-|\al|}_{\si(t)} f^{(\al)}(y))\, dt
	\\
	&= \int_{0}^{1}   \sum_{|\be| = p-|\al|} \frac{1}{\be!} \sum_{i=1}^d f^{(\al+\be+e_i)}(\si(t)) (y - \si(t))^\be \si_i'(t)\, dt
\end{align*}
By the Cauchy--Schwarz inequality,
\[
	|\sum_{i=1}^d f^{(\al+\be+e_i)}(\si(t)) (y - \si(t))^\be \si_i'(t)| 
	\le  |\nabla f^{(\al+\be)}(\si(t))| |\si'(t)| |y - \si(t)|^{|\be|}.  
\]	 	
Moreover, $|y - \si(t)| = |\si(1) - \si(t)| \le \ell(\si) (1-t)$. 
Thus 
\begin{align*}
	|(R^p_x f)^\al(y)| 
	&\le \sqrt d \sup_{\substack{\xi \in \si\\ |\ga| = p+1}} |f^{(\ga)}(\xi)| \, 
	\ell(\si)^{p+1-|\al|}
	 \int_{0}^{1}  \frac{(1-t)^{p-|\al|}}{(p-|\al|)!}  \, dt \sum_{|\be| = p-|\al|} \frac{|\be|}{\be!},
\end{align*}
that is \eqref{eq:DCjet2}.
\end{proof}

\begin{corollary}
	Let $M=(M_k)$ be a strongly regular weight sequence.
	For all $X \in \sH^1(\R^d)$ the functions in  
	$\cC^M(X)$ are precisely the functions which admit a $\cC^M$-extension to $\R^d$.
\end{corollary}

\begin{proof}
	This follows from \Cref{lem:extension} and the $\cC^M$-version of Whitney's extension theorem 
	\cite{Bruna80},
	since a bounded Lipschitz set is 
	quasiconvex, i.e., \eqref{m-regular} holds with $m=1$; 
	cf.\ \Cref{prop:alpharegular}.  
\end{proof}

\begin{corollary} 
	Let $M=(M_k)$ be a non-quasianalytic weight sequence of moderate growth 
	such that $M^a$ is non-quasianalytic for each $a>0$.
	Let $X \subseteq \R^d$ be a closed fat subanalytic subset. 
	Then the functions in $\widehat  \cC^{M}(X)$ are precisely the functions which admit a 
	$\widehat  \cC^{M}$-extension to $\R^d$. 
	If $X$ is simple, they are precisely the functions in $\widehat \cA^M(X)$. 
\end{corollary}

\begin{proof}
	This follows from \Cref{main:5}, \Cref{thm:subanreg}, and \Cref{lem:extension}. 
	Indeed, \Cref{lem:extension} implies that each $f \in \widehat \cC^{M}(X)$ defines a Whitney 
	jet of class $\widehat \cC^{M}$ on $X$ (the integer $m$ of \Cref{lem:extension} is local but it is 
	absorbed by $\widehat \cC^{M}$). 
	The extension theorem \cite[Theorem 8]{ChaumatChollet98} yields the required extension to $\R^d$.
\end{proof}

\subsection{Examples and counterexamples} \label{counterexamples}

The following examples complement the results and indicate their sharpness. 

\begin{example}[Infinitely flat fat cusps are not $\cA^\infty$-admissible]
	\label{flat}
Let $p : [0,\infty) \to [0,\infty)$ be a strictly increasing $\cC^\infty$-function which is infinitely flat at $0$. 
Consider the set $X := \{(x,y) \in \R^2 : x\ge 0,\, 0 \le y \le p(x)\}$ and the function $f : X \to \R$ defined by 
$f(x,y) = \sqrt{x^2 +y}$. Clearly, $f$ is $\cC^\infty$ in the interior of $X$ but $\p_y f$ does not extend continuously 
to the origin. 

On the other hand, $f \in \cA^\infty(X)$.
Let $x,y : \R \to \R$ be $\cC^\infty$-functions such that $(x(t),y(t)) \in X$ for all $t \in \R$. 
To see that $f \in \cA^\infty(X)$	
it suffices to prove 
that there is a $\cC^\infty$-function $z :\R \to \R$ such that $y= x^2 z$. 

We use the following result due to \cite[Theorem 7]{JorisPreissmann90}: 
{\it Let $\vh,\ps : \R \to \R$ be such that $\ps \in \cC^\infty$, $\vh \ps \in \cC^\infty$, and $|\vh| \le |\ps|^\al$ for some positive constant 
$\al$. Then $\vh \in \cC^{\lfloor 2\al \rfloor}$.}

We apply this result for $\ps = x^2$ and 
\[
	\vh = \begin{cases}
		y(t)/x(t)^2	& \text{ if } x(t) \ne 0,\\
		0 & \text{ if } x(t) = 0.
	\end{cases}
\]
The assumption $0 \le y \le p(x)$ implies that for each $n \in \N$ there is an interval $[0,\ep_n)$ such that 
for all $x \in [0,\ep_n)$ one has $y \le x^{2n+2}$. 
We may conclude that $\vh$ is $\cC^{2n}$ on the set $x^{-1}([0,\ep_n))$. Clearly, $\vh$ is $\cC^\infty$ on the set 
$\{t \in \R : x(t) \ne 0\}$. Thus $\vh$ is $\cC^\infty$ everywhere.
\end{example}

\begin{example}[Necessity of simpleness] \label{ex:simple}
	Let $X_1 = \{(x,y) \in \R^2 : x\ge 0, \, 0 \le y \le x\}$ and $X_2 = \{(x,y) \in \R^2 : 0 \le x \le y/2 \}$ 
	and set $X = X_1 \cup X_2$.
	The function $f$ on $X$ defined by $f(x,y) = x$ if $(x,y) \in X_1$ and $f(x,y) = y$ if $(x,y) \in X_2$ 
	belongs to $\cA^\infty(X)$ but clearly not to $\cC^\infty(X)$. 
	This follows from the fact that a $\cC^\infty$-curve $c(t)$ in $X$ must vanish of infinite order at each $t_0$ 
	with $c(t_0) \in X_1 \cap X_2 = \{0\}$. 
	Indeed, suppose that $c(t) = t^p \tilde c(t)$ with $(a,b) := \tilde c(0) \ne 0$ and $c(t) \in X_1$ if $t \le 0$ and  
	$c(t) \in X_2$ if $t \ge 0$. 
	If $p$ is even, it follows that $b \le a \le b/2$ which entails $a=b=0$, a contradiction.
	If $p$ is odd, we conclude that $0 \le a \le 0$, $b\le 0$, $a \le b/2$ hence $a=b=0$ again.   

	A modification of this example shows that the assumption that $X$ is simple cannot be replaced by the weaker 
	assumption that each $x \in X$ has a neighborhood $U$ such that $U \cap \interior(X)$ is connected: 
	Let $0<r<R$, consider $X := X_1 \cup X_2 \cup X_3$, where $X_3 = \{(x,y) \in \R^2 : x\ge0,\, y\ge0, x^2 + y^2 \ge R^2\}$, 
	and multiply $f$ with a smooth bump function which is $1$ on $B(0,r)$ and has support in $B(0,R)$. 
\end{example}

Nevertheless we have the following.

\begin{example} \label{ex:independent}
	Let $X_1 = \{(x,0) \in \R^2 : x\ge 0\}$ and $X_2 = \{(0,y) \in \R^2 : y \ge 0 \}$ 
	and set $X = X_1 \cup X_2$. Then $X$ is $\cA^\infty$-admissible.
	Indeed, let $f \in \cA^\infty(X)$.
	We may assume without loss of generality that $f(0,0) = 1$ (by multiplying with or adding a constant).
	Now $f|_{X_i}$ has a $\cC^\infty$-extension $F_i$ to $\R$ for $i=1,2$, 
	by \Cref{main:1}, and $F(x,y) := F_1(x)F_2(y)$ is a $\cC^\infty$-extension of $f$.
\end{example}

\begin{example}[There is no analogue for finite differentiability]
\label{ex:Glaeser}
This is an interesting consequence of \emph{Glaeser's inequality} \cite{Glaeser63R}: 
for $f : \R \to [0,\infty)$, 
\[
	f'(t)^2 \le 2 f(t) \|f''\|_{L^\infty(\R)}, \quad t \in \R. 
\] 
Indeed, consider the closed half-space $X = \{x \in \R^d : x_d \ge 0\}$ and the function $f : X \to \R$ 
given by $f(x) = x_d^{k+ 1/2}$. Then all partial derivatives of $f$ up to order $k$ extend continuously by $0$ 
to $\p X$, and the partial derivatives of order $k$ are $1/2$-H\"older continuous, but not better, near 
points of $\p X$.  
On the other hand, for each 
$\cC^{k,1}$-curve $c$ in $X$ with compact support, 
the composite $(f \o c)(t) = c_d(t)^{k+1/2}$ is $\cC^k$ with
\[
	(f\o c)^{(k)}(t) = C_k (c_d'(t))^k \sqrt{c_d(t)} +  D_k(c(t)),
\] 
where $t \mapsto D_k(c(t))$ is Lipschitz.
Since $\sqrt{c_d}$ is Lipschitz, by Glaeser's inequality, we conclude that $f \o c$ is of class $\cC^{k,1}$.
\end{example}

We want to add that the images of pseudo-immersions (which are not immersions) 
yield examples of sets $X \subseteq \R^d$ 
which are not $\cA^\infty$-admissible.

\begin{example}
	If $\on{gcd}(p,q) = 1$ then the map $\vh : \R \ni t \mapsto (t^p,t^q) \in \R^2$ is a pseudo-immersion, by 
	\cite{Joris82}, see also \cite{JorisPreissmann87}, \cite{JorisPreissmann90}, \cite{DuncanKrantzParks85}, and 
	\cite{AmemiyaMasuda89}. 
	Now the function $f(x,y) = y^{1/q}$ belongs to $\cA^\infty(\vh(\R))$ but has no smooth extension to $\R^2$.
\end{example}

The following example shows that there are closed fat sets $X \subseteq \R^d$ which 
satisfy 
\begin{equation} \label{notequal}
	\cA^\infty(X) = \big\{f : X \to \R : f \text{ satisfies \ref{lem:converse}(3)} \big\} \ne \cC^\infty(X).
\end{equation}

\begin{example} \label{ex:complementflatcusp}
	Let $X$ be the complement in $\R^2$ of the set $\{(x,y) \in \R^2 : x>0,\, |y| < e^{-1/x}\}$. 
	It is well-known (cf.\ \cite[Example 2.18]{Bierstone80a}) that
	there exist functions $f : X \to \R$ which satisfy \Cref{lem:converse}(3), but $f \not\in \cC^\infty(X)$. 
	 
	Let us show that for this $X$ the identity in \eqref{notequal} holds.
	To this end 
	let $h : \R \to \R$ be defined by $h (x) = 0$ if $x\le 0$ and 
	$h(x) = e^{-1/x}$ if $x >0$. Consider 
	\[
		X_\pm := \big\{(x,y) \in \R^2 : \pm y \ge h(x)\big\} \cup \big\{(x,y) \in \R^2 : x\le 0\big\}.
	\] 
	Then $X_\pm$ are $1$-sets and hence are $\cA^\infty$-admissible, 
	by \Cref{main:1}.
	
	Suppose $f \in \cA^\infty(X)$. Then $f$ is smooth on $\interior (X)$. 
	The restrictions $f|_{X_\pm}$ belong to $\cA^\infty(X_\pm)$, respectively.
	So all their derivatives extend to the boundary arcs $\{(x,y) \in \R^2 : x\ge0,\, \pm y = h(x)\}$,
	respectively.
	It remains to check that the extensions of the derivatives of $f|_{X_\pm}$ 
	coincide at the origin. 
	But this is clear, since they are uniquely determined by the restriction of $f$ to 
	$X_+ \cap X_-$.  

	For the converse suppose that $f : X \to \R$ satisfies \ref{lem:converse}(3). We have to show that  
	$f\o c$ is smooth for all smooth curves $c : \R \to X$. Since $X_\pm$ are $\cA^\infty$-admissible, this is 
	clear on the complement of $c^{-1}(0)$ in $\R$. Assume that $c(0) = 0$. 
	We claim that $f \o c$ is differentiable at $0$ and the chain rule $(f\o c)'(0) = f'(0)(c'(0))$ holds. 
	The set $X$ is star-shaped with respect to each point in $(-\infty,0]$.
	
	For each $v \in X$, the curve $\ga(t):= tv$ lies in $X$ for $0\le t\le 1$. 
	Moreover, $\ga_s(t) := \ga(t) + s^2 (-1 - \ga(t))$ lies in $X$ for $0\le t\le 1$ and $|s|\le 1$. 
	If $s \ne 0$, then $\ga_s(t) \in \interior(X)$ and hence
	\[
		\frac{f(\ga_s(t)) - f(\ga_s(0))}{t} = \int_0^1 (f\o \ga_s)'(tu)\, du = (1-s^2) \int_0^1 f'(\ga_s(tu))(v)\, du.
	\]  
	Letting $s \to 0$ and using \ref{lem:converse}(3), we get
	\[
		\frac{f(\ga(t)) - f(\ga(0))}{t} = \int_0^1 f'(\ga(tu))(v)\, du.
	\]
	This tends to $f'(\ga(0))(v)$ as $t \to 0$.

	Now for $0\le s \le 1$ and $t \in \R$ we have $s\cdot c(t) \in X$.  
	We may apply the last paragraph for $v = c(t)/t$ and obtain 
	\[
		\frac{f(c(t)) - f(0)}{t} = \int_0^1 f'(u c(t))\Big(\frac{c(t)}{t}\Big)\, du,
	\]
	which tends to $f'(0)(c'(0))$, since $f'(u c(t)) \to f'(0)$ uniformly on the bounded set $\{c(t)/t : t \text{ near }0\}$. 
	This proves the claim.

	By iteration we may conclude that $f \o c$ is smooth; cf.\ the proof of \cite[Theorem 24.5]{KM97}.
\end{example}



\def\cprime{$'$}
\providecommand{\bysame}{\leavevmode\hbox to3em{\hrulefill}\thinspace}
\providecommand{\MR}{\relax\ifhmode\unskip\space\fi MR }
\providecommand{\MRhref}[2]{%
  \href{http://www.ams.org/mathscinet-getitem?mr=#1}{#2}
}
\providecommand{\href}[2]{#2}

\end{document}